\documentclass[reqno,11pt]{amsart}

\usepackage{hyperref}
\usepackage{a4wide,amsthm,amssymb,mathrsfs,setspace,pstricks,booktabs,mathtools,amsmath,amsfonts}
\usepackage{graphicx}
\usepackage{xcolor}
\usepackage{array}
\newtheorem{theorem}{Theorem}[section]
\newtheorem{lemma}[theorem]{Lemma}
\newtheorem{proposition}[theorem]{Proposition}

\theoremstyle{definition}

\newtheorem{conjecture}[theorem]{Conjecture}

\theoremstyle{remark}
\newtheorem{remark}[theorem]{Remark}

\numberwithin{equation}{section}

\newcommand{\R}{\mathbb{R}}

\newcommand{\Rn}{\mathbb{R}^n}
\newcommand{\Sph}{\mathbb S}
\newcommand{\ip}[2]{\langle #1,#2\rangle}
\newcommand{\norm}[1]{\lVert #1\rVert}
\DeclareMathOperator{\supp}{supp}

\DeclareMathOperator{\tr}{tr}
\DeclareMathOperator{\spanop}{span}
\newcommand{\Id}{\mathrm{Id}}

\begin{document}
	
	\title [A sharp isoperimetric bound for the ratio of Dirichlet eigenvalues]{A sharp isoperimetric bound for the ratio of the fifth and fourth eigenvalues of Dirichlet Laplacian}

	\author{Zhen Song}    
	\address{Department of Mathematics, Northeastern University, Shenyang 110004, China}
	\email{songzhen@mail.neu.edu.cn}
	
	\author{Ruixiao Zhang}    
	\address{Department of Mathematics, Northeastern University, Shenyang 110004, China}
	\email{zhangrx@mail.neu.edu.cn}


	
	\begin{abstract}
		Let $\Omega\subset\Rn$ be a bounded domain with $n\geq2$. The well-known Payne-P\'olya-Weinberger conjecture for consecutive eigenvalue ratios asserts that
		$$
		\frac{\lambda_{m+1}(\Omega)}{\lambda_m(\Omega)}\leq\left(\frac{j_{n / 2,1}}{j_{n / 2-1,1}}\right)^2,\quad m=1,2,3,\cdots.
		$$
		Here, $j_{p,k}$ denotes the $k$-th positive zero of the Bessel function $J_p$, and $\lambda_j(\Omega)$ are the Dirichlet eigenvalues of $\Omega$, counted with multiplicity and ordered as
		$$
		0<\lambda_1(\Omega)<\lambda_2(\Omega)\leq\lambda_3(\Omega)\leq\cdots.
		$$
		The cases $m=1,2,3$ were established by Ashbaugh and Benguria in their seminal works. In this paper, we give an affirmative answer to the Payne-P\'olya-Weinberger conjecture for $m=4$ by combining a quadratic-form identity with a Brouwer degree argument. We also extend the corresponding spectral inequality to arbitrary bounded open sets in $\Rn$.
	\end{abstract}

	\maketitle
	
	\section{Introduction}\label{section1}
\renewcommand{\thethm}{\Alph{thm}}
Let $\Omega\subset\R^n$ be a bounded domain with $n\geq2$. Consider the classical Dirichlet eigenvalue problem for the Laplace operator on $\Omega$
\begin{equation}\label{dirichlet problem}
	\begin{cases}
		-\Delta u=\lambda u \quad &\hbox{in}\ \Omega, \\
		 u=0 \quad &\hbox{on}\ \partial\Omega, \\
	\end{cases}
\end{equation}
It is well-known that the spectrum of Dirichlet Laplacian consists of isolated eigenvalues of finite multiplicity accumulating to $\infty.$ Let $\{\lambda_k(\Omega)\}_{k=1}^\infty$ be the Dirichlet eigenvalues of $\Omega$, ordered so that 
$$
0<\lambda_1(\Omega)<\lambda_2(\Omega)\leq\lambda_3(\Omega)\leq\cdots,\quad \lambda_k(\Omega)\to\infty\quad\text{as}\quad k\to\infty,
$$
with a corresponding orthogonal set of eigenfunctions denoted by $\{u_k\}_{k=1}^\infty\subset H_0^1(\Omega)$. In this paper, we aim to find the optimal upper bound of the ratio $\lambda_5(\Omega)/\lambda_4(\Omega)$.

This problem can be traced back to the seminal work of Payne, P\'olya and Weinberger \cite{PaynePolyaWeinberger1956JMP}, where they proved in the two-dimensional case that the ratio of the first two Dirichlet eigenvalues satisfies
\begin{equation}\label{ppw inequality first}
	\frac{\lambda_2(\Omega)}{\lambda_1(\Omega)}\leq 3,\quad \forall\, \Omega\subset \R^2,
\end{equation}
and further conjectured that the optimal upper bound should be the value of $\lambda_2(\Omega)/\lambda_1(\Omega)$ on disk, i.e.,
\begin{equation}\label{ppw conjecture first}
	\frac{\lambda_2(\Omega)}{\lambda_1(\Omega)}\leq\left(\frac{j_{1,1}}{j_{0,1}}\right)^2=\frac{\lambda_2(\mathbb{B}^2)}{\lambda_1(\mathbb{B}^2)},\quad \forall\, \Omega\subset\R^2,
\end{equation}
where $\mathbb{B}^n$ is a ball in $\Rn$ and $j_{p,k}$ denotes the $k$-th positive zero of the Bessel function $J_p$. Later, Thompson \cite{Thompson1969SAM} generalized the inequality \eqref{ppw inequality first} to dimension $n$, obtaining that
\begin{equation}\label{ppw inequality in dimension n by thompson}
	\frac{\lambda_2(\Omega)}{\lambda_1(\Omega)}\leq 1+\frac4n,\quad\forall\,\Omega\subset\R^n,\quad n\geq2,
\end{equation}
and then extended the two-dimensional conjecture \eqref{ppw conjecture first} to general $n$-dimensional case, i.e.,
	\begin{equation}\label{ppw conjecture equation}
		\frac{\lambda_2(\Omega)}{\lambda_1(\Omega)}\leq\left(\frac{j_{n / 2,1}}{j_{n / 2-1,1}}\right)^2=\frac{\lambda_2(\mathbb{B}^n)}{\lambda_1(\mathbb{B}^n)},\quad\forall\,\Omega\subset\R^n,\quad n\geq2.
	\end{equation}
The truth of this conjecture means that one can determine whether a drum is circular or not by knowing its first two tones. Due to its profound significance, this conjecture has been widely studied by many mathematicians, such as Payne, P\'olya and Weinberger \cite{PaynePolyaWeinberger1955CR}, Brands \cite{Brands1964ARMA}, Chiti \cite{Chiti1981BUMIA,Chiti1983SIAM}, de Vries \cite{deVries1967ZNA}, Hile and Protter \cite{HileProtter1980Indiana}, since it was proposed, and was finally proved by Ashbaugh and Benguria in their seminal papers, see \cite{AshbaughBenguria1991BullAMS,AshbaughBenguria1992AnnMath,AshbaughBenguria1992CMP}. 

The general conjecture of Payne, P\'olya and Weinberger in this direction reads as follows.
\begin{conjecture}[Payne-P\'olya-Weinberger Conjecture] \label{ppw conjecture general}
	Let $\Omega\subset\R^n$ be a bounded domain with $n\geq2$. Then
	\begin{equation}\label{ppw conjecture equation general}
		\frac{\lambda_{m+1}(\Omega)}{\lambda_m(\Omega)}\leq\left(\frac{j_{n / 2,1}}{j_{n / 2-1,1}}\right)^2,\quad \text{for }m=1,2,3,\cdots.
	\end{equation}
\end{conjecture}
Combining the method used in \cite{AshbaughBenguria1992AnnMath} and an ingenious nodal domain argument, Ashbaugh and Benguria proved Conjecture \ref{ppw conjecture general} for $m=2$ and $m=3$ in \cite{AshbaughBenguria1991Duke} and \cite{AshbaughBenguria1993Edinburgh}, respectively. Also, the equality can never be attained if we require $\Omega$ to be a domain. It is worth mentioning that in \cite{AshbaughBenguria1993Edinburgh}, they actually proved that for any bounded domain $\Omega$ and any index $m\geq 2$ and $k\geq2$, 
\begin{equation}\label{ashbaugh benguria edin general}
	\frac{\lambda_{kl(m)}(\Omega)}{\lambda_m(\Omega)}\leq\sup_{\Omega'\subset\R^n}\frac{\lambda_{k}(\Omega')}{\lambda_1(\Omega')},
\end{equation}
where $l(m)$ denotes the number of nodal domains of the $m$-th Dirichlet eigenfunction $u_m$ of the Laplacian on $\Omega$. Since $l(2)=2$ for every bounded domain $\Omega$, we obtain
\begin{equation}\label{ashbaugh benguria edin}
	\frac{\lambda_{4}(\Omega)}{\lambda_2(\Omega)}\leq\sup_{\Omega'\subset\R^n}\frac{\lambda_{2}(\Omega')}{\lambda_1(\Omega')}=\left(\frac{j_{n / 2,1}}{j_{n / 2-1,1}}\right)^2,
\end{equation}
which implies the Payne-P\'olya-Weinberger Conjecture \ref{ppw conjecture general} for $m=2$ and $m=3$.

When $m\geq4$, Conjecture \ref{ppw conjecture general} is much more difficult to study. For example, when $m=4$, one may expect to prove this conjecture by using \eqref{ashbaugh benguria edin general} or the nodal domain argument in \cite{AshbaughBenguria1991Duke}. Therefore, we need to prove that $l(3)\geq 5/2$ or $l(4)\geq 5/2$. However, in three or higher dimensions, this condition does not hold in general, as is shown by the $n$-dimensional ball in which case $\lambda_2=\lambda_3=\cdots=\lambda_{n+1}$ and hence $l(2)=l(3)=\cdots=l(n+1)=2$. Thus the nodal domain argument used in \cite{AshbaughBenguria1993Edinburgh} does not work generally, and Conjecture \ref{ppw conjecture general} is still unsolved for $m\geq4$.

When $m=1$, the stronger Payne-P\'olya-Weinberger summation conjecture asserts that for any bounded domain $\Omega\subset \R^n$, $n\geq 2$, there holds
\begin{equation}\label{ppw summation}
\frac{\lambda_2(\Omega)+\cdots+\lambda_{n+1}(\Omega)}{\lambda_1(\Omega)}\leq n\left(\frac{j_{n / 2,1}}{j_{n / 2-1,1}}\right)^2=n\frac{\lambda_2(\mathbb{B}^n)+\cdots+\lambda_{n+1}(\mathbb{B}^n)}{\lambda_1(\mathbb{B}^n)}.
\end{equation}
Obviously, equation \eqref{ppw summation} is stronger than \eqref{ppw conjecture equation}. The corresponding conjecture is much more challenging than the one associated with \eqref{ppw conjecture equation} and has not yet been solved. We refer the interested reader to \cite{Brands1964ARMA,ChenZheng2011JDE,HileProtter1980Indiana,Marcellini1980JDE} and the references therein for recent progress on this conjecture. It is worth mentioning that Ashbaugh and Benguria proved in \cite{AshbaughBenguria1993SIAM} that
$$
\frac{1}{\lambda_2(\Omega)-\lambda_1(\Omega)}+\cdots+\frac{1}{\lambda_{n+1}(\Omega)-\lambda_1(\Omega)} \geq \frac{2 j_{n / 2-1}^2+n(n-4)}{6 \lambda_1(\Omega)}.
$$
By the asymptotic behaviors
$$
\frac{2 j_{n / 2-1}^2+n(n-4)}{6} \sim \frac{n^2}{4}\left[1+\frac{2}{3}(1.8557571) \frac{2^{2 / 3}}{n^{2 / 3}}-\frac{4}{n}+O(n^{-4 / 3})\right]
$$
and
$$
\frac{n}{(j_{n / 2} / j_{n / 2-1})^2-1} \sim \frac{n^2}{4}\left[1+\frac{2}{3}(1.8557571) \frac{2^{2 / 3}}{n^{2 / 3}}-\frac{2}{n}+O(n^{-4 / 3})\right],
$$
they conjectured in \cite{AshbaughBenguria1993SIAM} that
\begin{equation}\label{ab reciprocal gap}
	\frac{\lambda_1(\Omega)}{\lambda_2(\Omega)-\lambda_1(\Omega)}+\cdots+\frac{\lambda_1(\Omega)}{\lambda_{n+1}(\Omega)-\lambda_1(\Omega)} \geq \frac{n}{(j_{n / 2} / j_{n / 2-1})^2-1}
\end{equation}
with equality if and only if $\Omega$ is a ball in $\Rn$. This conjecture can be viewed as a weaken version of the summation conjecture associated with \eqref{ppw summation} in the sense that if \eqref{ppw summation} is valid, so is \eqref{ab reciprocal gap} by the inequality of arithmetic and geometric means. An important progress on this conjecture is that Wang and Xia \cite{WangXia2021AnalPDE} proved for any bounded smooth domain $\Omega$ that
\begin{equation}
	\frac{\lambda_1(\Omega)}{\lambda_2(\Omega)-\lambda_1(\Omega)}+\cdots+\frac{\lambda_1(\Omega)}{\lambda_{n}(\Omega)-\lambda_1(\Omega)} \geq \frac{n-1}{(j_{n / 2} / j_{n / 2-1})^2-1}
\end{equation}
with equality if and only if $\Omega$ is a ball in $\Rn$. This isoperimetric type inequality strongly supports the conjecture \eqref{ab reciprocal gap}. Very recently, Li, Tang and Zhang \cite{LiTangZhang2026} finally proved that every bounded domain $\Omega\subset \Rn$ satisfies
$$
\frac{\lambda_1(\Omega)}{\lambda_2(\Omega)-\lambda_1(\Omega)}+\cdots+\frac{\lambda_1(\Omega)}{\lambda_{n+1}(\Omega)-\lambda_1(\Omega)} \geq \frac{n}{(j_{n / 2} / j_{n / 2-1})^2-1}
$$
with equality if and only if $\Omega$ is a ball in $\Rn$ up to a zero capacity set, and hence confirmed the conjecture \eqref{ab reciprocal gap} proposed by Ashbaugh and Benguria.

In this paper, we study the Payne-P\'olya-Weinberger Conjecture \ref{ppw conjecture general} for $m=4$. The main result of this paper is as follows.
\begin{theorem}\label{main result}
	Let $\Omega$ be a bounded domain in $\Rn$, $n\geq2$. Then
	\begin{equation}\label{main result equation}
		\frac{\lambda_5(\Omega)}{\lambda_4(\Omega)}\leq\left(\frac{j_{n / 2,1}}{j_{n / 2-1,1}}\right)^2.
	\end{equation}
	Moreover, the constant in equation \eqref{main result equation} is optimal.
\end{theorem}
Theorem \ref{main result} settles the long-standing Payne-P\'olya-Weinberger Conjecture \ref{ppw conjecture general} for $m=4$. Furthermore, the inequality \eqref{main result equation} can be extended to arbitrary bounded open subsets of $\R^n$.
\begin{theorem}\label{main result for open set}
	Let $\Omega$ be a bounded open set in $\Rn$, $n\geq2$. Then
	\begin{equation}\label{main result equation open set}
		\frac{\lambda_5(\Omega)}{\lambda_4(\Omega)}\leq\left(\frac{j_{n / 2,1}}{j_{n / 2-1,1}}\right)^2.
	\end{equation}
	Moreover, the constant in equation \eqref{main result equation open set} is optimal.
\end{theorem}
As mentioned earlier, the difficulty in proving Theorem \ref{main result} lies in the fact that Courant's nodal domain theorem does not guarantee that the fourth eigenfunction $u_4$ has more than two nodal domains. To overcome this difficulty, we instead choose the test function from the linear space spanned by the first four eigenfunctions and use its positive and negative parts separately. More concretely, for a bounded domain $\Omega$, by rescaling $\Omega$ if necessary, we may assume that $\lambda_4(\Omega)=j_{n / 2-1,1}^2$. Suppose to the contrary that $\lambda_5(\Omega)>j_{n / 2,1}^2$. Define
\begin{equation}\nonumber
	q_\tau^\Omega(f)\coloneqq\int_\Omega(|\nabla f|^2-\tau f^2)\,dx,\quad f\in H_0^1(\Omega),
\end{equation}
and
$
V\coloneqq\operatorname{span}\{u_2,u_3,u_4\}
$.
Let $v\in S(V)\coloneqq\{\varphi\in V:\|\varphi\|_2=1\}$ and $f^\pm=(v-tu_1)^\pm$. For each $v$, choose $t\in\R$ such that both $f^+$ and $f^-$ are non-zero and satisfy
$$
q_{j_{n / 2-1,1}^2}^\Omega(f^+)=q_{j_{n / 2-1,1}^2}^\Omega(f^-)\leq0.
$$
Using the classical radial trial function given in \cite{AshbaughBenguria1992AnnMath} (which will be recalled in Section \ref{section3}), we construct two symmetric matrices. With the help of an energy identity, we may find a function with negative $q_{j_{n / 2,1}^2}^\Omega$-energy that is $L^2$-orthogonal to both $u_1$ and $v$. The projection of this function onto $V$ is non-zero and perpendicular to $v$. Letting $v$ vary over $S(V)\cong \mathbb{S}^2$ brings a topological obstruction into play, which will eventually lead to a contradiction with the assumption that $\lambda_5(\Omega)>j_{n / 2,1}^2$.

The rest of this paper is organized as follows. In Section \ref{section2}, we give some preliminary results that we will use throughout this paper. In Section \ref{section3}, we recall the classical radial trial function given by Ashbaugh and Benguria, present some properties of this function, and then establish an energy identity. In Section \ref{section4}, we balance the energies of the positive and negative parts of the test function and study the associated center equations. In Section \ref{section5}, by combining the energy estimates with a Brouwer degree argument, we prove Theorem \ref{main result}. Finally, the proof of Theorem \ref{main result for open set} will be given in Section \ref{section6}.

\section{Preliminaries}\label{section2}
First, we review some important facts that will be used repeatedly in this paper. For a real function $f$, set $f^+\coloneqq\max\{f,0\}$ and $f^-\coloneqq\max\{-f,0\}$. Therefore, $f=f^+-f^-$ and both the positive part $f^+$ and negative part $f^-$ are non-negative. Let $\Omega\subset \R^n$ be a bounded domain with $n\geq2$. The following lemma is basic and well-known.
\begin{lemma}\label{basic lemma for sobolev function}
	For any $f\in H^1(\Omega)$, we have $f^{\pm}\in H^1(\Omega)$ and
	\begin{equation}\label{nabla f+ and f-}
		\nabla f^+=\mathbf1_{\{f>0\}}\nabla f,
		\qquad \nabla f^-=-\mathbf1_{\{f<0\}}\nabla f,
		\quad\text{a.e. in }\Omega.
	\end{equation}
	The maps $f\mapsto f^\pm$ are continuous in $H^1$, and $f\in H_0^1(\Omega)$ implies $f^\pm\in H_0^1(\Omega)$. If $f\in H_0^1(\Omega)$ and $F\in W^{1,\infty}(\R^n)$, then $Ff\in H_0^1(\Omega)$. Moreover,
	\begin{equation}\label{product of f+ and f- is zero}
		f^+f^-=0,\quad \nabla f^+\cdot\nabla f^-=0,
		\quad\text{a.e. in }\Omega.
	\end{equation}
	Besides, for any $F,G\in W^{1,\infty}(\R^n)$, we also have
	\begin{equation}\label{product of Ff+ and Gf- is zero}
		(Ff^+)(Gf^-)=0,\quad \nabla (Ff^+)\cdot\nabla (Gf^-)=0,
		\quad\text{a.e. in }\Omega.
	\end{equation}
\end{lemma}
\begin{proof}
	For any $f\in H^1(\Omega)$, it is standard to show that $f^{\pm}\in H^1(\Omega)$ and that \eqref{nabla f+ and f-} holds by considering the function
	$$
	h_\varepsilon(z)\coloneqq\begin{cases}
		(z^2+\varepsilon^2)^{1/2}-\varepsilon, & z\geq0, \\
		0, & x< 0.
	\end{cases}
	$$
	Moreover, we also have $\nabla f=0$ a.e. on the set $\{f=0\}$, see \cite{Evans2010} for instance. Assume that $f_k\to f$ in $H^1$. Clearly, we have
	$$
	|f_k^+-f^+|\leq|f_k-f|,
	$$
	from which we know $f_k^+\to f^+$ in $L^2$. On the other hand, since
	\begin{equation}
		\nonumber
		\nabla f_k^+-\nabla f^+=\mathbf1_{\{f_k>0\}}(\nabla f_k-\nabla f)+(\mathbf1_{\{f_k>0\}}-\mathbf1_{\{f>0\}})\nabla f,
	\end{equation}
	up to a subsequence, we know $\nabla f_k^+\to \nabla f^+$ in $L^2$ by dominated convergence theorem. Therefore, we know the map $f\mapsto f^+$ is continuous in $H^1$. Similarly, one can prove that the map $f\mapsto f^-$ is also continuous in $H^1$. If we further assume that $f\in H_0^1(\Omega)$, by the definition of $H_0^1(\Omega)$, there exists $f_k\in C_0^\infty(\Omega)$ such that $f_k\to f$ in $H^1(\Omega)$. Obviously, $f_k^+$ is Lipschitz continuous and $\supp f_k^+\subset\subset \Omega$. By the basic properties of mollifier and convolution, we know $f_k^+\in H_0^1(\Omega)$. Besides, since
	$$
	|f_k^+|\leq |f_k|,\quad |\nabla f_k^+|=|\mathbf1_{\{f_k>0\}}\nabla f_k|\leq |\nabla f_k|,
	$$
	we know $\{f_k^+\}_{k=1}^\infty\subset H_0^1(\Omega)$ is a bounded sequence. Then, we may assume that $f_k^+\rightharpoonup h$ in $H_0^1(\Omega)$ for some $h\in H_0^1(\Omega)$. The continuity of the map $f\mapsto f^+$ implies that $h=f^+$, and hence $f^+\in H_0^1(\Omega)$. Since $f^-=(-f)^+$, we know $f^-\in H_0^1(\Omega)$. By a totally similar argument, one can prove that $Ff\in H_0^1(\Omega)$ for any $f\in H_0^1(\Omega)$ and $F\in W^{1,\infty}(\R^n)$. Finally, \eqref{product of f+ and f- is zero} and \eqref{product of Ff+ and Gf- is zero} can be verified directly by \eqref{nabla f+ and f-}.
\end{proof}
Recall that by the min-max principle, we know that the Dirichlet eigenvalues satisfy the following formula
\begin{equation}\label{min-max dirichlet eigenvalue}
	\lambda_k(\Omega)
	=
	\inf_{\substack{W\subset H_0^1(\Omega)\\ \dim W=k}}
	\ \sup_{0\ne f\in W}
	\frac{\int_\Omega|\nabla f|^2\,dx}
	{\int_\Omega f^2\,dx}.
\end{equation}
Moreover, the corresponding eigenfunctions $\{u_k\}_{k=1}^\infty\subset H_0^1(\Omega)$ form an orthogonal basis of $H_0^1(\Omega)$ and an orthonormal basis of $L^2(\Omega)$. For $\tau\in\R$ and $f,h\in H_0^1(\Omega)$, define
\begin{equation}\label{q-energy}
	q_\tau^{\Omega}(f,h)\coloneqq\int_\Omega\nabla f\cdot\nabla h\,dx-\tau\int_\Omega fh\,dx,
	\quad q_\tau^{\Omega}(f)\coloneqq q_\tau^{\Omega}(f,f).
\end{equation}
For any other open set $D\subset\R^n$, we similarly write
\[
q_\tau^{D}(f,h)=\int_D\nabla f\cdot\nabla h\,dx-\tau\int_D fh\,dx,
\quad q_\tau^{D}(f)=q_\tau^{D}(f,f),
\quad\forall f,h\in H_0^1(D).
\]
Let $\ip{f}{h}=\int_\Omega fh\,dx$, then one has
\begin{equation}\label{q-energy expansion}
	q_\tau^{\Omega}(f)=\sum_{j=1}^\infty(\lambda_j(\Omega)-\tau)|\ip{f}{u_j}|^2,
	\quad \forall f\in H_0^1(\Omega).
\end{equation}
\begin{lemma}\label{q-energy definite subspace}
	If $\lambda_k(\Omega)<\tau<\lambda_{k+1}(\Omega)$ and $W_k=\spanop\{u_1,\cdots,u_k\}$, then $q_\tau^{\Omega}$ is negative definite on $W_k$ and
	\begin{equation}\label{main:positivecomplement}
		q_\tau^{\Omega}(f)\ge(\lambda_{k+1}(\Omega)-\tau)\norm{f}_{L^2}^2,
		\quad\forall f\perp W_k.
	\end{equation}
	Furthermore, every linear subspace on which $q_\tau^{\Omega}$ is non-positive has dimension at most $k$.
\end{lemma}
\begin{proof}
	The assertions in this lemma follow from the spectral expansion \eqref{q-energy expansion} directly.
\end{proof}
Finally, we present the scaling rule for Dirichlet eigenvalues under domain dilations. For $s>0$, let $s\Omega\coloneqq\{sx:x\in\Omega\}$. Define
\[
f_s(y)\coloneqq s^{-n/2}f(y/s),\quad y\in s\Omega,\quad f\in H_0^1(\Omega).
\]
By a direct calculation, one can easily verify that this map preserves the $L^2$ norm and multiplies the Dirichlet energy by $s^{-2}$. Thus the following lemma holds.
\begin{lemma}
	For any $s>0$, there holds
	\begin{equation}\label{scaling eigenvalue}
		\lambda_j(s\Omega)=s^{-2}\lambda_j(\Omega),\quad j=1,2,3,\cdots.
	\end{equation}
\end{lemma}

\section{Trial function and a quadratic-form identity}\label{section3}
As we mentioned before, we can not simply select $u_4$ as the test function. Therefore, in this section, we will establish a quadratic-form energy identity for arbitrary functions in $H^1(\R^n)$ which will be used later to deduce a trace estimate needed in the proof of Theorem \ref{main result}.

From now on, let
\begin{equation}
	\nu=\frac n2-1,\quad\alpha=j_{\nu,1},\quad\beta=j_{\nu+1,1},\quad d=\beta^2-\alpha^2
\end{equation}
for simplicity. By Lemma \ref{scaling eigenvalue}, we may suppose, for contradiction and without loss of generality, that $\lambda_4(\Omega)=\alpha^2$ and $\lambda_5(\Omega)>\beta^2$ up to a domain dilation. Define the radial auxiliary function for $0<r<1$ by
\begin{equation}
	\phi(r)\coloneqq r^{-\nu}J_\nu(\alpha r),\quad
	\chi(r)\coloneqq r^{-\nu}J_{\nu+1}(\beta r).
\end{equation}
The functions $\phi$ and $\chi$ should be understood at $r=0$ by continuous extension and have been used to establish some other inequalities of Dirichlet eigenvalues, see \cite{LiTangZhang2026,WangXia2021AnalPDE} for instance. By the definition of Bessel function, we have
\begin{equation}
	\phi''+\frac{n-1}{r}\phi'+\alpha^2\phi=0
	\label{phi ode}
\end{equation}
and
\begin{equation}
	\chi''+\frac{n-1}{r}\chi'
	+\left(\beta^2-\frac{n-1}{r^2}\right)\chi=0.
	\label{chi ode}
\end{equation}
Both $\phi$ and $\chi$ are positive on $(0,1)$ and vanish at $r=1$. Also noteworthy, the function
$x\mapsto\phi(|x|)$ is the positive first Dirichlet eigenfunction of the unit
ball, with eigenvalue $\alpha^2$.

The strict spectral gap estimate and the monotonicity of the Bessel quotient given below are already  established in \cite{LiTangZhang2026}. We include proofs and present some finer results.
\begin{lemma}
	\label{bessel quotient basic lemma}
	For $n\ge2$, one has
	\begin{equation}
		d>n-1.
		\label{spectral gap alpha2 beta2}
	\end{equation}
	Define the Bessel quotient by
	\begin{equation}
		g(r)\coloneqq\frac{\chi(r)}{\phi(r)},\quad r\in(0,1).
	\end{equation}
	$g(r)$ has a continuous extension to $[0,1]$ with $g(0)=0$ and $g(1)=\gamma>0$. Furthermore, it satisfies
	\begin{equation}
		g''+\left(\frac{n-1}{r}+2\frac{\phi'}{\phi}\right)g'
		+\left(d-\frac{n-1}{r^2}\right)g=0,
		\quad 0<r<1,
		\label{g ode}
	\end{equation}
	and $g'(r)>0$ for $0<r<1$. More precisely, for some $c_0>0$, we have
	\begin{equation}
		g(r)=c_0r+O(r^3),\quad r\downarrow0,
		\label{g origin behavior}
	\end{equation}
	\begin{equation}
		g(r)=\gamma\left(1-\frac{d-n+1}{6}(1-r)^2
		+O((1-r)^3)\right),\quad r\uparrow1,
		\label{g boundary behavior}
	\end{equation}
	and
	\begin{equation}
		g'(r)=\frac{\gamma(d-n+1)}3(1-r)+O((1-r)^2),\quad r\uparrow1.
		\label{gprime boundary behavior}
	\end{equation}
	Extend $g$ by the constant $\gamma$ on $[1,\infty)$.  For $A\in\R^n$, define
	\begin{equation}
		F_A(x)\coloneqq
		\begin{cases}
			g(|x-A|)\dfrac{x-A}{|x-A|},&x\ne A,\\
			0,&x=A,
		\end{cases}
		\label{ana:F-definition}
	\end{equation}
	and write $F_A=(F_{A,1},\cdots,F_{A,n})$ explicitly. Then $F_A\in C^1(\R^n;\R^n)$, $|F_A|\le\gamma$, and the operator norm of its gradient is bounded by a constant independent of $A$. Moreover, both $F_A$ and $DF_A$ depend continuously on $A$ in the uniform norm.
\end{lemma}

\begin{proof}
	\textbf{Step 1: Proof of \eqref{spectral gap alpha2 beta2}.}  Since $r^{n-1}\chi\chi'$ vanishes at $r=0$ and $r=1$, multiplying \eqref{chi ode} by $\chi r^{n-1}$ and then integrating by parts gives
	\begin{equation}
		\beta^2=
		\frac{\int_0^1\chi'(r)^2r^{n-1}\,dr}
		{\int_0^1\chi(r)^2r^{n-1}\,dr}
		+(n-1)
		\frac{\int_0^1\chi(r)^2r^{n-3}\,dr}
		{\int_0^1\chi(r)^2r^{n-1}\,dr}.
		\nonumber
	\end{equation}
	Recall that $\alpha^2$ is the first Dirichlet eigenvalue of the unit ball. By the variational principle and the fact that $r^{-2}>1$ in $(0,1)$, we obtain
	\[
	\beta^2>\alpha^2+n-1.
	\]
	
	\textbf{Step 2: Proof of \eqref{g origin behavior}--\eqref{gprime boundary behavior}.} Recall that the trial function $g(r)$ has been thoroughly studied by Ashbaugh and Benguria \cite{AshbaughBenguria1992AnnMath}, and is shown there to be continuously differentiable on $[0,\infty)$. We refer the reader to \cite{AshbaughBenguria1992AnnMath} for the original derivation. By the power series expansion of the Bessel functions near $r=0$ (see \cite{NISTDLMF10} for instance), we have
	\begin{equation}
		\phi(r)=A_0\left(1-\frac{\alpha^2}{2n}r^2+O(r^4)\right),
		\quad A_0=\frac{\alpha^\nu}{2^\nu\Gamma(\nu+1)}>0,
	\end{equation}
	and
	\begin{equation}
		\chi(r)=D_0r\left(1-\frac{\beta^2}{2(n+2)}r^2+O(r^4)\right),
		\quad D_0=\frac{\beta^{\nu+1}}{2^{\nu+1}\Gamma(\nu+2)}>0.
	\end{equation}
	By taking $c_0=D_0/A_0$, we obtain \eqref{g origin behavior}.
	
	Recall that if a solution $u$ to a second-order homogeneous linear ordinary differential equation satisfies $u(1) = u^\prime(1) = 0$, then $u$ vanish identically. Since $\phi$ and $\chi$ are positive on the interval $(0,1)$, thus we have
	$\phi'(1)<0$ and $\chi'(1)<0$. Let
	\[
	A_1=-\phi'(1)>0,\quad D_1=-\chi'(1)>0,
	\quad \gamma=D_1/A_1.
	\]
	By differentiating the corresponding two equations with respect to $r$ and taking $r=1$, we arrive at
	$$
	\phi'''(1)=\bigl(n(n-1)-\alpha^2\bigr)\phi'(1),\quad
	\chi'''(1)=\bigl(n^2-1-\beta^2\bigr)\chi'(1),
	$$
	and
	\begin{equation}\nonumber
		\phi''(1)=-(n-1)\phi'(1),\quad\chi''(1)=-(n-1)\chi'(1).
	\end{equation}
	By change of variable $\rho=1-r$, the corresponding Taylor expansions are
	\begin{equation}
		\phi(1-\rho)
		=A_1\cdot\rho\left(1+\frac{n-1}{2}\rho
		+\frac{n(n-1)-\alpha^2}{6}\rho^2+O(\rho^3)\right)
	\end{equation}
	\begin{equation}
		\chi(1-\rho)
		=D_1\cdot\rho\left(1+\frac{n-1}{2}\rho
		+\frac{n^2-1-\beta^2}{6}\rho^2+O(\rho^3)\right).
	\end{equation}
	Letting $\rho\to 0$ yields \eqref{g boundary behavior}.	Note that the linear term $1-r$ in \eqref{g boundary behavior} vanishes, 
	which implies \eqref{gprime boundary behavior} immediately. In particular, we also have
	$g'(1)=0$.
	
	\textbf{Step 3: $g$ is monotonic and satisfies \eqref{g ode}.} 
	Substituting $\chi=\phi g$ in \eqref{chi ode} and using
	\eqref{phi ode} gives \eqref{g ode}. Equivalently, we have
	\begin{equation}
		\left(r^{n-1}\phi(r)^2g'(r)\right)'
		=\left(\frac{n-1}{r^2}-d\right)
		r^{n-1}\phi(r)^2g(r).
		\label{g ode general}
	\end{equation}
	Let $H(r)=r^{n-1}\phi(r)^2g'(r)$. The asymptotic behaviors of $g^\prime$ and $\phi$ show that
	$H(0)=H(1)=0$. Thanks to \eqref{spectral gap alpha2 beta2}, we know $r_0=\sqrt{(n-1)/d} \in (0,1)$.  The right-hand side of
	\eqref{g ode general} is strictly positive for $0<r<r_0$ and strictly
	negative for $r_0<r<1$.  Thus
$$
	H(r)=\int_0^r H'(s)\,ds>0,\quad\text{for }0<r\le r_0,
$$
and
$$
	H(r)=-\int_r^1 H'(s)\,ds>0,\quad\text{for }r_0\le r<1.
$$
	Consequently, $g'(r)>0$ for $0<r<1$.
	
	\textbf{Step 4: Properties of $F_A$.} Recall that $F_A$ is just the vector-valued trial map in
	\cite{LiTangZhang2026}. For $x\ne A$, denote $r=|x-A|$ and $\theta=(x-A)/r$. Differentiating $F_A$ away from the center gives
	\begin{equation}
		DF_A(x)=g'(r)\theta\theta^{\mathsf T}
		+\frac{g(r)}r\left(\Id-\theta\theta^{\mathsf T}\right),\quad x\neq A.
		\label{gradient of FA}
	\end{equation}
	As $x\to A$, \eqref{g origin behavior} gives
	\[
	F_A(x)=c_0(x-A)+O(|x-A|^3),
	\]
	and hence $DF_A(A)=c_0\Id$.
	Moreover, by the power series expansion of the Bessel functions near $r=0$, we have
	$$g'(r)=c_0+O(r^2).$$
	Together with \eqref{gradient of FA}, we obtain that
	$$
	DF_A(x)=c_0\Id+O(r^2),\quad x\to A,
	$$
	from which we know that $DF_A$ is continuous at $A$.
	
	As $r\to1-$, we have $g(r)\to\gamma$ and $g'(r)\to0$.
	Therefore,
	$$
	DF_A(A+r\theta)\longrightarrow
	\gamma(\Id-\theta\theta^{\mathsf T})
	$$
	uniformly for $\theta\in\Sph^{n-1}$.
	For $r>1$, the definition of $g$ gives
	$$
	DF_A(x)=\frac{\gamma}{r}
	\left(\Id-\theta\theta^{\mathsf T}\right).
	$$
	Summing up the above, we know $F_A\in C^1(\R^n;\R^n)$. Since $g$ is non-decreasing, we immediately have $|F_A(x)|\leq\gamma$ for all $x\in\mathbb R^n$. Since $g'(r)$ and $g(r)/r$ are bounded for $r>0$, the formula for $DF_A$ gives
	$$
	\|DF_A(x)\|_{\mathrm{op}}\leq C,
	\quad
	C\coloneqq\max\left\{\sup_{r>0}|g'(r)|,\sup_{r>0}\frac{|g(r)|}{r}\right\}<\infty,
	$$
	first for $x\ne A$ and then, by continuity, also for $x=A$.
	In particular, $C$ is independent of $A$.
	
	Notice that for $|x|>1$, the definition of $g$ gives
	$$
	DF_0(x)=\frac{\gamma}{|x|}
	\left(I-\frac{xx^{\mathsf T}}{|x|^2}\right).
	$$
	The matrix in parentheses is an orthogonal projection, so
	$$
	\|DF_0(x)\|_{\mathrm{op}}\leq\frac{\gamma}{|x|}\longrightarrow0\quad\text{as }|x|\to\infty.
	$$ 
	Since $DF_0$ is continuous on $\Rn$ and tends to zero at infinity, it is uniformly continuous on $\Rn$. Finally, the bound on $DF_0$ implies that $F_0$ is globally Lipschitz continuous. Since $F_A(x)=F_0(x-A)$, we obtain that
	$$
	\|F_A-F_B\|_\infty=\sup_{y\in\mathbb{R}^n}|F_0(y)-F_0(y+A-B)|\leq C|B-A|\longrightarrow0
	\quad\text{as }B\to A.
	$$
	Similarly, since $DF_A(x)=DF_0(x-A)$, the uniform continuity of $DF_0$ yields
	$$
	\begin{aligned}
		\|DF_A-DF_B\|_\infty=\sup_{y\in\mathbb{R}^n}\|DF_0(y)-DF_0(y+A-B)\|_{\mathrm{op}}\rightarrow0\quad\text{as }B\to A.
	\end{aligned}
	$$
	Here, the uniform norms are taken over $\mathbb R^n$, using the Euclidean norm for vectors and the operator norm for matrices. The proof is complete.
\end{proof}

Next, we derive a quadratic-form energy identity that will be used to control the traces of two trial matrices in Section \ref{section5}. 

For $\tau\in\R$ and $f,h\in H^1(\R^n)$, define the quadratic-form for the whole space as follows:
$$
q_\tau^{\R^n}(f,h)\coloneqq\int_{\R^n}\nabla f\cdot\nabla h\,dx-\tau\int_{\R^n}fh\,dx,
\quad q_\tau^{\R^n}(f)\coloneqq q_\tau^{\R^n}(f,f).
$$
Notice that if $\widetilde f,\widetilde h\in H^1(\Rn)$ are the zero extensions of $f,h\in H_0^1(\Omega)$, then
\[
q_\tau^{\R^n}(\widetilde f,\widetilde h)=q_\tau^{\Omega}(f,h),\quad q_\tau^{\R^n}(\widetilde f)=q_\tau^{\Omega}(f),
\]
where $q_\tau^{\Omega}$ is the quadratic-form define by \eqref{q-energy}.

Following Li, Tang and Zhang\cite{LiTangZhang2026}, we define the weight $p$ as
\begin{equation}\nonumber
	p(r)\coloneqq \gamma^2-g(r)^2.
\end{equation}
Then $p>0$ on $[0,1)$ and $p=0$ on $[1,\infty)$. By Lemma \ref{bessel quotient basic lemma}, we know
\begin{equation}
	p(r)=\frac{\gamma^2(d-n+1)}3(1-r)^2+O((1-r)^3),
	\quad
	\frac{\phi'(r)}{\phi(r)}=-\frac1{1-r}+O(1),
	\quad r\uparrow 1.
	\label{p expansion}
\end{equation}
When $u$ is the zero extension of an $L^2$-normalized first Dirichlet eigenfunction with eigenvalue $\alpha^2$, the left-hand side of the identity \eqref{energy identity Rn} given in the next proposition is exactly the scalar deficit $\mathfrak{D}$ defined in \cite[(5.8)]{LiTangZhang2026}. Proposition \ref{energy identity prop} reformulates and extends their radial remainder estimate in two respects. First, it gives an exact identity on $\R^n$, including both the interior and exterior remainders. In the first-eigenfunction setting, the interior remainder equals the right-hand side of \cite[(4.15)]{LiTangZhang2026} plus a non-negative integral involving the radial derivative. Dropping this additional integral recovers their inequality. Second, in Section 4 of \cite{LiTangZhang2026}, it is assumed that $u\in H^1(\Rn)$ has compact support and satisfies
$$
\Delta(u^2)=2|\nabla u|^2-2\alpha^2u^2
\quad\text{in }\mathcal D'(\R^n).
$$
Our identity holds for every $u\in H^1(\R^n)$, with neither assumption required. The term $q_{\alpha^2}^{\R^n}(u)$ remains on the left-hand side and need not vanish.

\begin{proposition}[The quadratic-form energy identity]
	\label{energy identity prop}
	For every $A\in\R^n$ and $u\in H^1(\R^n)$, there holds
	\begin{equation}
		\begin{aligned}
			\gamma^2q_{\alpha^2}^{\R^n}(u)-\sum_{i=1}^nq_{\beta^2}^{\R^n}(F_{A,i}u)&=\int_{B_1(A)}p(r)\phi(r)^2\left|\nabla\left(\frac{u(x)}{\phi(r)}\right)\right|^2\,dx\\
			&\quad+\gamma^2\int_{\R^n\setminus B_1(A)}
			\left(d-\frac{n-1}{r^2}\right)u(x)^2\,dx,
			\quad r=|x-A|.
			\label{energy identity Rn}
		\end{aligned}
	\end{equation}
	In particular, for every $u\in H_0^1(\Omega)$,
	\begin{equation}
		\sum_{i=1}^nq_{\beta^2}^{\Omega}(F_{A,i}u)\le\gamma^2q_{\alpha^2}^{\Omega}(u).
		\label{energy identity omega}
	\end{equation}
\end{proposition}
\begin{proof}
	By making the change of variable $x=A+y$, with $u$ and $F_{A,i}$ replaced by $u(A+\cdot)$ and $F_{0,i}$, respectively, the quadratic-forms are preserved. Therefore, it suffices to establish the identity for $A=0$. For simplicity, we write
	$$
	F_i=F_{0,i},\quad r = |x|,\quad \theta = \frac{x}{r},\quad\text{for }r\ne 0.
	$$
	
	\textbf{Step 1: Reduction to the weighted energy.} For $u\in C_c^\infty(\R^n)$, \eqref{gradient of FA} gives
	\begin{equation*}
		\frac{\partial F_i}{\partial x_j}
		=\left(g'(r)-\frac{g(r)}r\right)\theta_i\theta_j
		+\frac{g(r)}r\delta_{ij}.
	\end{equation*}
	By a basic calculation, we have
	$$
	\sum_{i=1}^nF_i^2=g^2,\quad\sum_{i=1}^nF_i\nabla F_i=gg'\theta,\quad \sum_{i=1}^n|\nabla F_i|^2=g'^2+\frac{n-1}{r^2}g^2.
	$$
	Combining the above identities with the identity $\nabla(F_i u)=F_i\nabla u+u\nabla F_i$ and then summing over $i$ yields
	\begin{equation}
		\sum_{i=1}^nq_{\beta^2}^{\R^n}(F_i u)=\int_{\R^n}\biggl[g^2|\nabla u|^2+2gg'u\,\partial_ru+\left(g'^2+\frac{n-1}{r^2}g^2-\beta^2g^2\right)u^2\biggr]\,dx.
		\label{q beta expansion}
	\end{equation}
	Recall that the radial function $x\mapsto p(|x|)$ belongs to $C^1(\R^n)$, vanishes outside the unit ball, and has bounded second derivatives on each side of $\partial B_1$.	In particular, $p(1)=p'(1)=0$. Hence, for any $\varphi\in C_c^\infty(\R^n)$, integration by parts on $B_1$ gives
	\begin{equation*}
		\int_{\R^n}p\,\Delta\varphi\,dx
		=\int_{B_1}(\Delta p)\varphi\,dx
		+\int_{\partial B_1}\bigl(p\,\partial_r\varphi-p'\varphi\bigr)\,dS
		=\int_{B_1}(\Delta p)\varphi\,dx.
	\end{equation*}
	Since $p'=-2gg'$, integration by parts again yields
	\[
	\int_{\R^n}-2gg'u\,\partial_ru\,dx
	=\int_{\R^n}\nabla p\cdot u\nabla u\,dx
	=-\frac12\int_{\R^n}(\Delta p)u^2\,dx.
	\]
	Together with \eqref{q beta expansion}, we rewrite
	\begin{equation}
		\gamma^2q_{\alpha^2}^{\R^n}(u)-\sum_{i=1}^nq_{\beta^2}^{\R^n}(F_i u)=\int_{\R^n}\left(p|\nabla u|^2+W(r)u^2\right)\,dx,
		\label{qbeta qalpha difference}
	\end{equation}
	where
	\begin{equation}
		W=\beta^2g^2-\alpha^2\gamma^2-g'^2-\frac{n-1}{r^2}g^2
		-\frac12\Delta p.
		\label{W def}
	\end{equation}
	Note that $W(1)$ does not affect the integral, since the sphere has Lebesgue measure zero and $\Delta p$ has no singularity on $\partial B_1$. Besides, it is not hard to verify that
	\[
	-\frac12\Delta p
	=g'^2+gg''+\frac{n-1}{r}gg',\quad 0<r<1,
	\]
	and
	$$g=\gamma,\quad p=0,\quad r>1.$$
	We may then simplify $W(r)$ as follows
	\begin{equation}\label{W new}
		W(r) = 
		\begin{cases}
			-\alpha^2p+p'\frac{\phi'}{\phi},\,\, &0<r<1,\\
			\gamma^2\left(d-\frac{n-1}{r^2}\right),\,\, &r>1.
		\end{cases}
	\end{equation}
	
	\textbf{Step 2: \eqref{energy identity Rn} holds for smooth function with compact support.} First, we fix $0<\varepsilon<1$ and expand the square on $B_{1-\varepsilon}$:
	\begin{align*}
		\int_{B_{1-\varepsilon}}
		p\phi^2\left|\nabla\frac u\phi\right|^2\,dx
		={}&\int_{B_{1-\varepsilon}}
		\left(p|\nabla u|^2
		-p\nabla(u^2)\cdot\frac{\nabla\phi}{\phi}
		+p u^2\frac{|\nabla\phi|^2}{\phi^2}\right)\,dx\\
		={}&\int_{B_{1-\varepsilon}}
		\left(p|\nabla u|^2
		+u^2\left[\frac{\nabla p\cdot\nabla\phi}{\phi}
		+p\frac{\Delta\phi}{\phi}\right]\right)\,dx\\
		&-\int_{\partial B_{1-\varepsilon}}
		p u^2\frac{\phi'}{\phi}\,dS\\
		={}&\int_{B_{1-\varepsilon}}
		\left(p|\nabla u|^2+
		\left[-\alpha^2p+p'\frac{\phi'}{\phi}\right]u^2\right)\,dx
		-\int_{\partial B_{1-\varepsilon}}
		p u^2\frac{\phi'}{\phi}\,dS.
	\end{align*}
	By \eqref{p expansion}, we know	$p\phi'/\phi=O(1-r)$ near $r=1$. Letting $\varepsilon\to0$, we obtain that
	\begin{equation}
		\int_{B_1}\left(p|\nabla u|^2+\left[-\alpha^2p+p'\frac{\phi'}{\phi}\right]u^2\right)\,dx=\int_{B_1}p\phi^2\left|\nabla\frac u\phi\right|^2\,dx.
		\label{first term expansion}
	\end{equation}
	Combining \eqref{qbeta qalpha difference}, \eqref{W new}, and \eqref{first term expansion} proves \eqref{energy identity Rn} for smooth function $u$ with compact support.
	
	\textbf{Step 3: \eqref{energy identity Rn} holds for functions belong to $H^1(\R^n)$.} On $B_1$, define the operator $\mathcal{T}$ by
	$$
	\mathcal T u\coloneqq\sqrt p\,\nabla u-\sqrt p\,\frac{\nabla\phi}{\phi}u.
	$$
	From the regularity, positivity of $\phi$ and the asymptotic
	\[
	\sqrt p=O(1-r)\quad\text{and}\quad \frac{\phi'}{\phi}=O((1-r)^{-1})
	\]
	near the sphere, we conclude that both coefficients of $\mathcal{T}$ are bounded. Thus $\mathcal T:H^1(\R^n)\to L^2(B_1;\R^n)$ is bounded. Since $\phi$ is smooth and positive in $B_1$, one has $u/\phi\in H^1_{\mathrm{loc}}(B_1)$ and
	\begin{equation*}
		\mathcal T u=\sqrt p\,\phi\,\nabla(u/\phi)\quad\text{a.e. in }B_1.
	\end{equation*}
	Therefore,
	$$
	\int_{B_1}p\phi^2\left|\nabla\left(\frac u\phi\right)\right|^2\,dx=\|\mathcal T u\|_{L^2(B_1)}^2\leq C\|u\|_{H^1(\R^n)}^2.
	$$
	It is worthy pointing out that the weighted integral is finite even without requiring $u/\phi\in H^1(B_1)$.
	
	Choose $u_j\in C_c^\infty(\R^n)$ such that $u_j\to u$ in $H^1(\R^n)$. The boundedness of $\mathcal T$ and Lemma \ref{bessel quotient basic lemma} imply
	\begin{equation}
		\mathcal T u_j\rightarrow\mathcal T u
		\quad\text{in }L^2(B_1;\R^n)
	\end{equation}
	and
	$$
	F_i u_j\rightarrow F_i u\quad\text{in }H^1(\R^n),\quad i=1,\cdots,n.
	$$
	The integral over $B_1$ on the right-hand side of \eqref{energy identity Rn} equals to $\|\mathcal T u_j\|_{L^2(B_1;\R^n)}^2$ and hence converges to the corresponding integral for $u$. Moreover,
	$$
	\|u_j^2-u^2\|_{L^1(\R^n)}\le\bigl(\|u_j\|_{L^2(\R^n)}+\|u\|_{L^2(\R^n)}\bigr)\|u_j-u\|_{L^2(\R^n)}\rightarrow0.
	$$
	Since the coefficient $\gamma^2\bigl(d-(n-1)/r^2\bigr)$ is bounded for $r\ge1$, the integral over $\R^n\setminus B_1$ on the right-hand side of \eqref{energy identity Rn} also converges to the corresponding integral for $u$.
	
	Finally, the strong $H^1$ convergence of $u_j$ and $F_i u_j$ implies convergence of all quadratic-forms on the left-hand side of \eqref{energy identity Rn}. Therefore, applying the identity to $u_j$ and passing to the limit proves the desired identity for every $u\in H^1(\R^n)$.

	\textbf{Step 4: Proof of \eqref{energy identity omega}}. By \eqref{spectral gap alpha2 beta2}, the coefficient satisfies $d-(n-1)/r^2\ge d-(n-1)>0$ for $r\ge1$. It is clear that 
	\[
	\sum_{i=1}^nq_{\beta^2}^{\R^n}(F_{A,i}v)\le\gamma^2q_{\alpha^2}^{\R^n}(v),\quad v\in H^1(\R^n).
	\]
	Fix $u\in H_0^1(\Omega)$ and let $\widetilde u$ denote its zero extension. Then, $\widetilde u\in H^1(\R^n)$ and $\nabla\widetilde u=\mathbf1_\Omega\nabla u$ almost everywhere. Each $F_{A,i}$ is bounded and Lipschitz, so	$F_{A,i}u\in H_0^1(\Omega)$ by Lemma \ref{basic lemma for sobolev function}. Consequently, we have
	$$
		q_{\alpha^2}^{\R^n}(\widetilde u)=q_{\alpha^2}^{\Omega}(u)
	$$
	and
	$$
	q_{\beta^2}^{\R^n}(F_{A,i}\widetilde u)=q_{\beta^2}^{\Omega}(F_{A,i}u),\quad i=1,\cdots,n.
	$$
	The proof is complete.
\end{proof}

\section{Energy balance and the choice of centers}\label{section4}

Define
\[
V\coloneqq\spanop\{u_2,u_3,u_4\},\quad S(V)\coloneqq\{\varphi\in V:\norm{\varphi}_{L^2(\Omega)}=1\}.
\]
Thus $S(V)$ is a two-dimensional sphere. For each $v\in S(V)$, we first choose $t\in\R$ such that the positive and negative parts of 
$v-tu_1$ have equal $q_{\alpha^2}^{\Omega}$-energy. We then find centers $A_+$ and $A_-$ satisfying
\[
q_{\beta^2}^{\Omega}(F_{A_\pm,i}(v-tu_1)^\pm,(v-tu_1)^\pm)=0,
\quad i=1,\cdots,n.
\]
Set
\begin{equation}\label{definition of e original}
	e=e(v,t)\coloneqq v-tu_1,\quad
	f_+\coloneqq e^+,\quad f_-\coloneqq e^-,\quad
	G(v,t)\coloneqq q_{\alpha^2}^{\Omega}(f_+)-q_{\alpha^2}^{\Omega}(f_-).
\end{equation}
Here $e^+\coloneqq\max\{e,0\}$ and $e^-\coloneqq\max\{-e,0\}$, so that both $f_+$ and $f_-$ are non-negative. Thus, $f_+$ and $f_-$ have equal $q_{\alpha^2}^{\Omega}$-energy precisely when $G(v,t)=0$. The following lemma shows that such a parameter $t$ exists for every $v\in S(V)$ and that all possible choices of $t$ are uniformly bounded.

\begin{lemma}\label{equal energy lemma}
	There exists a positive constant $S>0$ such that
	\begin{equation}
		G(v,t)>0,\quad\text{for }\,\,t\geq S,\qquad	G(v,t)<0,\quad\text{for }\,\,t\leq-S,
		\label{threshold for G}
	\end{equation}
uniformly for $v\in S(V)$. The set
	\begin{equation}
		K=\{(v,t)\in S(V)\times\R:G(v,t)=0\}
		\label{balance set K}
	\end{equation}
	is compact and lies in $S(V)\times(-S,S)$. Moreover, for every $v\in S(V)$, $K$ has a non-empty fiber $K_v\coloneqq\{t\in \R:(v,t)\in K\}$. At every point of $K$, there holds
	\begin{equation}
		f_+\ne0,\quad f_-\ne0,\quad
		q_{\alpha^2}^{\Omega}(f_+)=q_{\alpha^2}^{\Omega}(f_-)=\frac12q_{\alpha^2}^{\Omega}(e)\leq0.
		\label{energy balance formula}
	\end{equation}
	Consequently, there are constants $r_*>0$ and $\mu_*=dr_*>0$ such that for every $(v,t)\in K$,
	\begin{equation}
			\norm{f_\pm(v,t)}_{L^2(\Omega)}^2\geq r_*,\quad \mu_\pm(v,t)\coloneqq -q_{\beta^2}^{\Omega}(f_\pm(v,t))\geq\mu_*.
		\label{mass lower bound}
	\end{equation}
	Moreover, the inner products $\ip{u_1}{f_\pm}$ also have a uniform positive	lower bound on $K$.
\end{lemma}
\begin{proof}
	Recall that the eigenfunctions are $L^2$-orthonormal, $u_1>0$ almost everywhere in $\Omega$, and $\lambda_1<\alpha^2$ while $\lambda_j\le\alpha^2$ for $j=2,3,4$.

	\textbf{Step 1: Continuity of the parameter-dependent functions.} For $w=\sum_{j=2}^4c_ju_j\in V$, it is easy to find out that
	$$
	\norm{w}_{H^1(\Omega)}^2=\sum_{j=2}^4(1+\lambda_j)c_j^2\leq(1+\alpha^2)\norm{w}_{L^2(\Omega)}^2.
	$$
	Consequently, for $(v,t),(\widetilde v,\widetilde t)\in S(V)\times\R$,
	\begin{equation}
		\begin{aligned}
			\norm{e(v,t)-e(\widetilde v,\widetilde t)}_{H^1(\Omega)}\leq\sqrt{1+\alpha^2}\norm{v-\widetilde v}_{L^2(\Omega)}+|t-\widetilde t|\norm{u_1}_{H^1(\Omega)}.
		\end{aligned}\label{continuity of e}
	\end{equation}
	Then, we know that the map $(v,t)\mapsto e(v,t)$ is continuous and therefore, the maps $(v,t)\mapsto f_+(v,t)$ and $(v,t)\mapsto f_-(v,t)$ are also  continuous by Lemma \ref{basic lemma for sobolev function}. The quadratic-form is continuous as well. Indeed, for any $h,k\in H_0^1(\Omega)$,
	$$
	\begin{aligned}
		\bigl|q_{\alpha^2}^{\Omega}(h)-q_{\alpha^2}^{\Omega}(k)\bigr|\leq\bigl(\norm{\nabla h}_{L^2}+\norm{\nabla k}_{L^2}\bigr)\norm{\nabla(h-k)}_{L^2}+\alpha^2\bigl(\norm{h}_{L^2}+\norm{k}_{L^2}\bigr)\norm{h-k}_{L^2}.
	\end{aligned}
	$$
	It follows that $G$ is jointly continuous in $(v,t)$. Notice that the following estimate holds
	\begin{equation}
		\sup_{v\in S(V)}\norm{v}_{H^1(\Omega)}\leq\sqrt{1+\alpha^2}.\label{estimate for v}
	\end{equation}
	
	\textbf{Step 2: The sum of the $q_{\alpha^2}^{\Omega}$-energy.} Since $e=f_+-f_-$ and
	$$
	f_+f_-=0,\quad \nabla f_+\cdot\nabla f_-=0,\quad\text{a.e. in }\Omega,
	$$
	we have $q_{\alpha^2}^{\Omega}(f_+,f_-)=0$ and hence
	\[
	q_{\alpha^2}^{\Omega}(e)=q_{\alpha^2}^{\Omega}(f_+)+q_{\alpha^2}^{\Omega}(f_-).
	\]
	Furthermore, since $v\perp u_1$ in $L^2(\Omega)$, we know
	\[
	q_{\alpha^2}^{\Omega}(v,u_1)=(\lambda_1-\alpha^2)\ip{v}{u_1}=0.
	\]
	Writing $v=\sum_{j=2}^4c_ju_j$, we also have
	\[
	q_{\alpha^2}^{\Omega}(v)
	=\sum_{j=2}^4(\lambda_j-\alpha^2)c_j^2\le0.
	\]
	Using $\norm{u_1}_{L^2(\Omega)}=1$, we therefore obtain for every $(v,t)\in S(V)\times\R$ that
	\begin{equation}
			q_{\alpha^2}^{\Omega}(f_+)+q_{\alpha^2}^{\Omega}(f_-)=q_{\alpha^2}^{\Omega}(e)=q_{\alpha^2}^{\Omega}(v)-(\alpha^2-\lambda_1)t^2\leq0.
		\label{energy split}
	\end{equation}

	\textbf{Step 3: Uniform sign for large positive and negative $t$.} For $t>0$, we know by \eqref{estimate for v} that as $|t|\to\infty$,
	$$
	\sup_{v\in S(V)}\left\|{\frac{e(v,t)}t+u_1}\right\|_{H^1(\Omega)}=\sup_{v\in S(V)}\frac{\norm{v}_{H^1(\Omega)}}t\leq\frac{\sqrt{1+\alpha^2}}t\rightarrow0.
	$$
	For $\varepsilon>0$, by Lemma \ref{basic lemma for sobolev function}, we know there exists $\eta>0$, such that if $\|h+u_1\|_{H^1(\Omega)}\leq \eta$, then
	$$
	\norm{h^+}_{H^1(\Omega)}=\norm{h^+-(-u_1)^+}_{H^1(\Omega)}<\varepsilon,\quad \norm{h^--u_1}_{H^1(\Omega)}=\norm{h^--(-u_1)^-}_{H^1(\Omega)}<\varepsilon.
	$$
	Here, we have used the fact that $u_1$ is non-negative in $\Omega$. Since for $t>0$, we have
	$$
	\left(\frac{e(v,t)}t\right)^\pm=\frac{f_\pm(v,t)}t,
	$$
	thus the above argument shows that
	\[
	\frac{f_+(v,t)}t\rightarrow0\quad\text{in }H_0^1(\Omega),\quad\frac{f_-(v,t)}t\rightarrow u_1\quad\text{in }H_0^1(\Omega),
	\]
	uniformly for $v\in S(V)$. By the homogeneity of the quadratic-form,
	\[
	\frac{G(v,t)}{t^2}=q_{\alpha^2}^{\Omega}\left(\frac{f_+(v,t)}t\right)-q_{\alpha^2}^{\Omega}\left(\frac{f_-(v,t)}t\right).
	\]
	The continuity of the quadratic-form implies that
	\begin{equation}
		\frac{G(v,t)}{t^2}\rightarrow-q_{\alpha^2}^{\Omega}(u_1)=\alpha^2-\lambda_1>0,\quad t\to+\infty,
		\label{limit of G/t^2}
	\end{equation}
	uniformly for $v\in S(V)$.
	
	For negative $t$, write $t=-s$ with $s>0$, then
	$$
	\frac{e(v,-s)}s=\frac vs+u_1,\quad\sup_{v\in S(V)}\left\|\frac{e(v,-s)}s-u_1\right\|_{H^1(\Omega)}
	\le\frac{\sqrt{1+\alpha^2}}s.
	$$
	By a similar argument, we have
	$$
	\frac{f_+(v,-s)}s\rightarrow u_1\quad\text{in }H_0^1(\Omega),\quad \frac{f_-(v,-s)}s\rightarrow0
	\quad\text{in }H_0^1(\Omega),
	$$
	uniformly in $v$. Since $s^2=t^2$, we obtain that
	\begin{equation}
		\frac{G(v,t)}{t^2}\rightarrow q_{\alpha^2}^{\Omega}(u_1)=\lambda_1-\alpha^2<0,\quad t\to-\infty,
		\label{limit of G/t^2 negative}
	\end{equation}
	uniformly for $v\in S(V)$. Therefore, we can choose $S\ge1$, independent of $v\in S(V)$, such that
	$$
		G(v,t)\ge\frac{\alpha^2-\lambda_1}{2}\,t^2>0\quad\text{for }t\ge S
	$$
	and
	$$
	G(v,t)\le-\frac{\alpha^2-\lambda_1}{2}\,t^2<0\quad \text{for }t\le-S.
	$$

	\textbf{Step 4: Nonemptiness and compactness of $K$.} Fix $v\in S(V)$. We know that the function $t\mapsto G(v,t)$ is continuous on $[-S,S]$, and
	$$
	G(v,-S)<0<G(v,S).
	$$
	Therefore, by the intermediate value theorem, there exists $t\in(-S,S)$ such that $G(v,t)=0$, i.e., $K_v$ is nonempty. Obviously, by \eqref{threshold for G}, we have
	$$
	K=\{(v,t)\in S(V)\times\R:G(v,t)=0\}\subset S(V)\times(-S,S),
	$$
	and we may also write
	\begin{equation}
		K=\{(v,t)\in S(V)\times[-S,S]:G(v,t)=0\}. \label{new description of K}
	\end{equation}
	Since $V$ is finite-dimensional, its $L^2$-unit sphere $S(V)$ is compact. Hence $S(V)\times[-S,S]$ is compact. By \eqref{new description of K} and the continuity of $G$, we immediately know that $K$ is compact.

	\textbf{Step 5: Energy balance and non-vanishing of $f_\pm$.} Fix $(v,t)\in K$. Since $G(v,t)=0$, the definition of $G$ gives
	$$
	q_{\alpha^2}^{\Omega}(f_+)=q_{\alpha^2}^{\Omega}(f_-).
	$$
	Together with \eqref{energy split}, we have
	\begin{equation}
		q_{\alpha^2}^{\Omega}(f_+)=q_{\alpha^2}^{\Omega}(f_-)=\frac12q_{\alpha^2}^{\Omega}(e)\le0.\label{energy identity used in the proof of Lem 4}
	\end{equation}
	Now we show that neither $f_+$ nor $f_-$ vanishes. Suppose that one of them were zero. Then its energy would be zero and hence
	\[
	q_{\alpha^2}^{\Omega}(f_+)=q_{\alpha^2}^{\Omega}(f_-)=0,
	\]
	where we have used \eqref{energy identity used in the proof of Lem 4}. Therefore, by \eqref{energy split},
	\[
	0=q_{\alpha^2}^{\Omega}(e)
	=q_{\alpha^2}^{\Omega}(v)-(\alpha^2-\lambda_1)t^2
	\le-(\alpha^2-\lambda_1)t^2.
	\]
	Since $\alpha^2>\lambda_1$, the only possibility is that $t=0$. Thus, we have $e=v$. If $f_+=0$, then $v\le0$ almost everywhere. If $f_-=0$, then $v\ge0$ almost everywhere. In either case, $v$ has only one sign. Since $v\perp u_1$, we obtain that
	\[
	0=|\ip{v}{u_1}|
	=\int_\Omega |v|u_1\,dx.
	\]
	Since $u_1>0$ almost everywhere, we know $v=0$ almost everywhere, thus contradicting to the fact that $\norm{v}_{L^2(\Omega)}=1$.

	\textbf{Step 6: Uniform positive lower bounds on $K$.} The maps $(v,t)\mapsto f_\pm(v,t)$ are continuous in $L^2(\Omega)$. Therefore, the two functions
	\[
	(v,t)\mapsto \norm{f_+(v,t)}_{L^2(\Omega)}^2,
	\quad
	(v,t)\mapsto \norm{f_-(v,t)}_{L^2(\Omega)}^2
	\]
	are continuous on $K$. By Step 5, we know both $\norm{f_+(v,t)}_{L^2(\Omega)}^2$ and $\norm{f_-(v,t)}_{L^2(\Omega)}^2$ are strictly positive on $K$. Since $K$ is nonempty and compact, each function attains a strictly positive minimum. Therefore, we can find a positive constant $r_*>0$ such that
	\[
	\norm{f_\pm(v,t)}_{L^2(\Omega)}^2\ge r_*\quad\text{for every }(v,t)\in K.
	\]
    Also by Step 5, we know $q_{\alpha^2}^{\Omega}(f_\pm(v,t))\le0$ on $K$. Since $\beta^2=\alpha^2+d$, we obtain that
	\[
	\begin{aligned}
		\mu_\pm(v,t)&=-q_{\beta^2}^{\Omega}(f_\pm(v,t))\\
		&=-q_{\alpha^2}^{\Omega}(f_\pm(v,t))+d\norm{f_\pm(v,t)}_{L^2(\Omega)}^2\\
		&\ge d\norm{f_\pm(v,t)}_{L^2(\Omega)}^2\ge dr_*.
	\end{aligned}
	\]
	Thus $\mu_*=dr_*>0$ gives the required uniform lower bound in \eqref{mass lower bound}.
	
	It remains to bound $\ip{u_1}{f_\pm(v,t)}$ from below. For each $(v,t)\in K$, the functions $f_\pm(v,t)$ are non-negative and non-zero, so both $f_+$ and $f_-$ are positive on a set of positive measure. Since $u_1>0$ almost everywhere, we know
	\begin{equation}\label{positive inner product}
		\ip{u_1}{f_\pm(v,t)}=\int_\Omega u_1f_\pm(v,t)\,dx>0.
	\end{equation}
	These inner products are continuous in $(v,t)$. Indeed, Cauchy-Schwarz inequality shows that
	\[
	\begin{aligned}
		\left|\ip{u_1}{f_\pm(v,t)}-\ip{u_1}{f_\pm(\widetilde v,\widetilde t)}\right|\leq		\norm{u_1}_{L^2(\Omega)}\norm{f_\pm(v,t)-f_\pm(\widetilde v,\widetilde t)}_{L^2(\Omega)},
	\end{aligned}
	\]
	and the last factor tends to zero as $(\widetilde v,\widetilde t)\to(v,t)$. The desired uniform positive lower bound of $\ip{u_1}{f_\pm(v,t)}$ then follows by \eqref{positive inner product}, the continuity of $(v,t)\mapsto \ip{u_1}{f_\pm(v,t)}$ and the compactness of $K$.
\end{proof}

We now turn to the choice of centers. For $(v,t)\in S(V)\times\R$ and $A\in\R^n$, define
\begin{equation}
		M_{\pm,i}(v,t,A)\coloneqq q_{\beta^2}^{\Omega}(F_{A,i}f_\pm,f_\pm),\quad 1\leq i\leq n,\label{moment M}\\
\end{equation}
and
\begin{equation}
	Q_{\pm,ij}(v,t,A)\coloneqq q_{\beta^2}^{\Omega}(F_{A,i}f_\pm,F_{A,j}f_\pm),\quad 1\leq i,j\leq n.\label{matrix Q}
\end{equation}
Here, $M_\pm$ is a vector in $\R^n$, and $Q_\pm$ is a real symmetric $n\times n$ matrix. Our aim is to find two centers $A_\pm$ satisfying that $M_\pm(v,t,A_\pm)=0$. Note that the matrices $Q_\pm$ record the values of the quadratic-form on linear combinations of the functions
$F_{A,i}f_\pm$.

To this end, we first express $M_\pm$ and $Q_\pm$ as integrals. For $v=\sum_{j=2}^4c_ju_j$ and $\tau\in\R$, define
\begin{equation}
	h_\tau(v,t)\coloneqq\sum_{j=2}^4c_j(\lambda_j-\tau)u_j-t(\lambda_1-\tau)u_1\in L^2(\Omega).
	\label{h tau}
\end{equation}
Since $e=v-tu_1$, for any $w\in H^1_0(\Omega)$, we have
\begin{equation}
	q_\tau^{\Omega}(e,w)=\int_\Omega h_\tau w\,dx.
	\label{qtau e w}
\end{equation}
Define
\begin{equation}
	\eta_+(v,t)\coloneqq-f_+h_{\beta^2},\quad\eta_-(v,t)\coloneqq f_-h_{\beta^2}.
	\label{eta pm}
\end{equation}
Both functions belong to $L^1(\Omega)$ by the Cauchy-Schwarz inequality. Note that $\eta_\pm$ need not to be non-negative. The subscripts used here indicate only whether they are constructed from $f_+$ or $f_-$. We use the term total signed mass for $\int_\Omega\eta_\pm\,dx$.

The next lemma shows that
$$
M_\pm(v,t,A)=-\int_\Omega F_A\eta_\pm\,dx.
$$
Thus the required center conditions $M_\pm(v,t,A_\pm)=0$ are equivalent to
\[
\int_\Omega F_{A_\pm}\eta_\pm\,dx=0.
\]
In \cite{LiTangZhang2026}, the same radial field is used but the center is chosen such that
$\int_\Omega F_Au^2\,dx=0$, see \cite[(5.2)]{LiTangZhang2026} for instance. Here, the non-negative weight $u^2$ is replaced by the signed functions $\eta_\pm$. 

\begin{lemma}\label{lemma of M and Q}
	For every $(v,t,A)\in S(V)\times\R\times\R^n$, we have
	\begin{equation}
		\int_\Omega\eta_\pm\,dx=-q_{\beta^2}^{\Omega}(f_\pm)=\mu_\pm(v,t),
		\label{total mass}
	\end{equation}
	\begin{equation}
		M_{\pm,i}(v,t,A)=-\int_\Omega F_{A,i}\eta_\pm\,dx,
		\label{M integral}
	\end{equation}
	\begin{equation}
		Q_{\pm,ij}(v,t,A)=-\int_\Omega F_{A,i}F_{A,j}\eta_\pm\,dx+\int_\Omega(\nabla F_{A,i}\cdot\nabla F_{A,j})f_\pm^2\,dx.
		\label{Q integral}
	\end{equation}
	Moreover,
	\begin{equation}
		G(v,t)=\int_\Omega |e(v,t)|h_{\alpha^2}(v,t)\,dx.
		\label{G integral}
	\end{equation}
	The maps $\eta_\pm:S(V)\times\R\rightarrow L^1(\Omega)$	are continuous. Moreover, the maps $	M_\pm:S(V)\times\R\times\R^n\rightarrow\R^n$ and $Q_\pm:S(V)\times\R\times\R^n\rightarrow\R^{n\times n}$ are jointly continuous in $(v,t,A)$.
	
	For every $(v,t)\in K$,
	$$
	\int_\Omega\eta_\pm(v,t)\,dx\geq\mu_*>0.
	$$
	Besides, there exists a constant $B>0$, independent of $(v,t)\in K$, such that
	\begin{equation}
		\norm{\eta_\pm(v,t)}_{L^1(\Omega)}\leq B,\quad (v,t)\in K.
		\label{upper bound of eta}
	\end{equation}
\end{lemma}

\begin{proof}
	\textbf{Step 1: Verification of integral identities.} Let $\zeta\in W^{1,\infty}(\R^n)$. By Lemma \ref{basic lemma for sobolev function}, $\zeta f_\pm\in H_0^1(\Omega)$. By taking $\tau=\beta^2$ and $w=\zeta f_\pm$ in \eqref{qtau e w}, we obtain that
	\[
	q_{\beta^2}^{\Omega}(e,\zeta f_\pm)=\int_\Omega \zeta f_\pm h_{\beta^2}\,dx.
	\]
	Recall that $e=f_+-f_-$ and that
	\[
	f_+f_-=0,\quad f_+\nabla f_-=f_-\nabla f_+=0,\quad\nabla f_+\cdot\nabla f_-=0,\quad\text{a.e. in }\Omega.
	\]
	By a direct calculation, we may verify that $q_{\beta^2}^{\Omega}(f_-,\zeta f_+)=0$. Therefore, we have
	\[
	\begin{aligned}
		q_{\beta^2}^{\Omega}(e,\zeta f_+)
		=q_{\beta^2}^{\Omega}(f_+,\zeta f_+),\quad q_{\beta^2}^{\Omega}(e,\zeta f_-)		&=-q_{\beta^2}^{\Omega}(f_-,\zeta f_-).
	\end{aligned}
	\]
	Summing up the above, we get
	$$
	q_{\beta^2}^{\Omega}(f_+,\zeta f_+)=\int_\Omega\zeta f_+h_{\beta^2}\,dx=-\int_\Omega\zeta\eta_+\,dx
	$$
	and
	$$
	q_{\beta^2}^{\Omega}(f_-,\zeta f_-)=-\int_\Omega\zeta f_-h_{\beta^2}\,dx=-\int_\Omega\zeta\eta_-\,dx,
	$$
	where we have used the fact that $\eta_+=-f_+h_{\beta^2}$ and $\eta_-=f_-h_{\beta^2}$, or equivalently,
	\begin{equation}
		q_{\beta^2}^{\Omega}(f_\pm,\zeta f_\pm)=-\int_\Omega\zeta\eta_\pm\,dx.
		\label{q beta truncated}
	\end{equation}
	The previously established bounds on $F_A$ and $DF_A$ in Lemma \ref{bessel quotient basic lemma} imply that each $F_{A,i}$ and each product $F_{A,i}F_{A,j}$ belongs to $W^{1,\infty}(\R^n)$. By taking $\zeta=1$ and $\zeta=F_{A,i}$ in \eqref{q beta truncated}, we get \eqref{total mass} and \eqref{M integral}. The product rule gives
	\[
	\begin{aligned}
		q_{\beta^2}^{\Omega}(F_{A,i}f_\pm,F_{A,j}f_\pm)=q_{\beta^2}^{\Omega}
		(f_\pm,F_{A,i}F_{A,j}f_\pm)+\int_\Omega	(\nabla F_{A,i}\cdot\nabla F_{A,j})f_\pm^2\,dx.
	\end{aligned}
	\]
	Taking $\zeta=F_{A,i}F_{A,j}$ gives \eqref{Q integral}. By a totally similar argument, one has
	\[
	q_{\alpha^2}^{\Omega}(f_+)
	=\int_\Omega f_+h_{\alpha^2}\,dx,
	\quad
	q_{\alpha^2}^{\Omega}(f_-)
	=-\int_\Omega f_-h_{\alpha^2}\,dx.
	\]
	Subtracting these equalities and using $f_++f_-=|e|$ gives \eqref{G integral}.
	
	\textbf{Step 2: The continuity of $\eta_\pm$, $M_\pm$ and $Q_\pm$.} Let
	\[
	(v_k,t_k,A_k)\rightarrow(v,t,A)
	\quad\text{in }S(V)\times\R\times\R^n.
	\]
	For simplicity, we write
	\[
		f_{\pm,k}=f_\pm(v_k,t_k),\quad f_\pm=f_\pm(v,t),\quad h_k=h_{\beta^2}(v_k,t_k),\quad h=h_{\beta^2}(v,t).
	\]
	Obviously, we know
	\[
	h_k\rightarrow h\quad\text{in }L^2(\Omega),\quad f_{\pm,k}\rightarrow f_\pm
	\quad\text{in }L^2(\Omega).
	\]
	By the definitions of $\eta_\pm$ and the Cauchy-Schwarz	inequality,
	\[
	\begin{aligned}
		\norm{\eta_\pm(v_k,t_k)-\eta_\pm(v,t)}_{L^1(\Omega)}\leq\norm{f_{\pm,k}-f_\pm}_{L^2(\Omega)}
		\norm{h_k}_{L^2(\Omega)}+\norm{f_\pm}_{L^2(\Omega)}\norm{h_k-h}_{L^2(\Omega)}
		\rightarrow0.
	\end{aligned}
	\]
	Similarly,
	\[
	\norm{f_{\pm,k}^2-f_\pm^2}_{L^1(\Omega)}\leq\bigl(\norm{f_{\pm,k}}_{L^2(\Omega)}+\norm{f_\pm}_{L^2(\Omega)}\bigr)\norm{f_{\pm,k}-f_\pm}_{L^2(\Omega)}
	\rightarrow0.
	\]
	Therefore, both $\eta_\pm$ and $f_\pm^2$ depend continuously on $(v,t)$ in $L^1(\Omega)$. Since $A_k\to A$, the continuity of $F_A$ and $DF_A$ in the uniform norm (see Lemma \ref{bessel quotient basic lemma}) gives
	\[
	\norm{F_{A_k}-F_A}_\infty\rightarrow0,\quad	\norm{DF_{A_k}-DF_A}_\infty\rightarrow0.
	\]
	Consequently, each of
	\[
	F_{A_k,i},\quad	F_{A_k,i}F_{A_k,j},\quad\nabla F_{A_k,i}\cdot\nabla F_{A_k,j}
	\]
	converges uniformly to the corresponding expression with
	center $A$. Recall that the following elementary estimate holds: if $a_k\to a$ in $L^\infty(\Omega)$
	and $b_k\to b$ in $L^1(\Omega)$, then
	\[
	\begin{aligned}
		\left|\int_\Omega a_kb_k\,dx-\int_\Omega ab\,dx\right|\leq\norm{a_k-a}_{L^\infty(\Omega)}
		\norm{b_k}_{L^1(\Omega)}+\norm{a}_{L^\infty(\Omega)}\norm{b_k-b}_{L^1(\Omega)}
		\rightarrow0.
	\end{aligned}
	\]
	Applying this estimate to each integral in \eqref{M integral} and \eqref{Q integral}, we obtain that
	\[
	\begin{aligned}
		M_\pm(v_k,t_k,A_k)\rightarrow M_\pm(v,t,A)\quad\text{and}\quad 
		Q_\pm(v_k,t_k,A_k)\rightarrow Q_\pm(v,t,A).
	\end{aligned}
	\]
	This proves the joint continuity in $(v,t,A)$.
	
	\textbf{Step 3: The desired bounds on $K$.} The $L^1$ continuity of $\eta_\pm$ proved above implies that the two functions $(v,t)\mapsto\norm{\eta_\pm(v,t)}_{L^1(\Omega)}$ are continuous. Since $\eta_\pm(v,t)\in L^1(\Omega)$ for each $(v,t)$ and $K$ is compact, we can find a positive constant $B>0$ such that \eqref{upper bound of eta} holds. By Lemma \ref{equal energy lemma}, we know at every $(v,t)\in K$, there holds
	\[
	q_{\alpha^2}^{\Omega}(f_\pm)\leq0,\quad\norm{f_\pm}_{L^2(\Omega)}^2\geq r_*.
	\]
	Therefore, by \eqref{total mass} and $\beta^2=\alpha^2+d$, we have
	\[
	\begin{aligned}
		\int_\Omega\eta_\pm\,dx=-q_{\beta^2}^{\Omega}(f_\pm)=-q_{\alpha^2}^{\Omega}(f_\pm)
		+d\norm{f_\pm}_{L^2(\Omega)}^2\geq d\norm{f_\pm}_{L^2(\Omega)}^2
		\geq dr_*=\mu_*>0.
	\end{aligned}
	\]
	Combining the two bounds, we obtain that
	\[
	\mu_*\leq\int_\Omega\eta_\pm\,dx\leq\int_\Omega|\eta_\pm|\,dx=\norm{\eta_\pm}_{L^1(\Omega)}\leq B,
	\quad\forall\, (v,t)\in K.
	\]
	Although the total mass of $\eta_\pm$ is positive, we can not conclude that $\eta_\pm$ is non-negative almost everywhere in $\Omega$.
\end{proof}

The solution of the balance equation $G(v,t)=0$ has already been restricted to $S(V)\times(-S,S)$. We now try to find a uniform bound on the centers: for $(v,t)\in K$, we will show that the equations $M_\pm(v,t,A_\pm)=0$ can not hold when $|A_\pm|$ is too large. We first give an estimate for general functions $f\in H_0^1(\Omega)$.

\begin{remark}\label{general estimate on center}
	For $f\in H_0^1(\Omega)$, define
	$$
	M_f(A)\coloneqq\bigl(q_{\beta^2}^{\Omega}(F_{A,i}f,f)\bigr)_{i=1}^n.
	$$
	By a direct calculation, we have
	\begin{equation}
		M_f(A)
		=\int_\Omega F_A\bigl(|\nabla f|^2-\beta^2f^2\bigr)\,dx
		+\int_\Omega f\,DF_A\nabla f\,dx.
		\label{MfA direct expansion}
	\end{equation}
	Choose $D>0$ such that $\Omega\subset B_D(0)$. Fix a center $A$ with $r=|A|>D+1$, and put $\omega=A/r$. For each $x\in\Omega$, write $\rho=|x-A|$. Then we know $\rho\geq r-|x|>r-D>1$ and hence
	\[
	F_A(x)=\gamma\frac{x-A}{\rho}.
	\]
	Firstly, since
	\[
	\frac{x-A}{\rho}+\frac{A}{r}
	=\frac{x}{\rho}
	+\left(\frac1r-\frac1\rho\right)A,
	\]
	and $|\rho-r|\leq|x|$, we obtain that
	\[
	\begin{aligned}
		|F_A(x)+\gamma\omega|
		\leq\gamma\frac{|x|}{\rho}+\gamma r\left|\frac1r-\frac1\rho\right|=\gamma\frac{|x|+|\rho-r|}{\rho}\leq\frac{2\gamma|x|}{\rho}
		\leq\frac{2\gamma D}{r-D}.
	\end{aligned}
	\]
	Now, by \eqref{gradient of FA}, we know
	\[
	DF_A(x)=\frac{\gamma}{\rho}\left(I-\theta\theta^{\mathsf T}\right),\quad\theta=\frac{x-A}{\rho}.
	\]
	Since $|\theta|=1$, for every $c\in\R^n$, we know
	\[
	\left|\left(I-\theta\theta^{\mathsf T}\right)c\right|^2=|c|^2-(c\cdot\theta)^2\leq|c|^2.
	\]
	Therefore, by the definition of the Euclidean operator norm,
	\[
	\norm{DF_A(x)}_{\mathrm{op}}
	=\sup_{|c|=1}|DF_A(x)c|
	\leq\frac{\gamma}{\rho}
	\leq\frac{\gamma}{r-D}.
	\]
	Summing up the above, we conclude that
	\[
	\sup_{x\in\Omega}|F_A(x)+\gamma\omega|
	\leq\frac{2\gamma D}{r-D},
	\qquad
	\sup_{x\in\Omega}\norm{DF_A(x)}_{\mathrm{op}}
	\leq\frac{\gamma}{r-D}.
	\]
	
	Now, to identify the main term in $M_f(A)$, rewrite \eqref{MfA direct expansion} as
	\[
	\begin{aligned}
		M_f(A)+\gamma\omega q_{\beta^2}^{\Omega}(f)
		&=\int_\Omega(F_A+\gamma\omega)
		\bigl(|\nabla f|^2-\beta^2f^2\bigr)\,dx+\int_\Omega f\,DF_A\nabla f\,dx.
	\end{aligned}
	\]
	The preceding estimates and Cauchy-Schwarz inequality yield that
	\begin{equation}
		\begin{aligned}
			\left|M_f(A)+\gamma\omega q_{\beta^2}^{\Omega}(f)\right|&\leq\frac{2\gamma D}{r-D}\left(\norm{\nabla f}_{L^2(\Omega)}^2+\beta^2\norm f_{L^2(\Omega)}^2\right)\\
			&\quad+\frac{\gamma}{r-D}\norm f_{L^2(\Omega)}\norm{\nabla f}_{L^2(\Omega)}.
		\end{aligned}
		\label{estimate for Mf + qbeta}
	\end{equation}
	For $r\geq2D+2$, the order of the right-hand side is at most $Cr^{-1}\norm f_{H^1(\Omega)}^2$, where $C$ depends only on $D$, $\beta^2$ and $\gamma$. Thus, 
	\[
	M_f(A)\cdot\omega\geq-\gamma q_{\beta^2}^{\Omega}(f)-\frac C r\norm f_{H^1(\Omega)}^2.
	\]
	If $q_{\beta^2}^{\Omega}(f)<0$, then the first term on the right-hand side is positive, while the second term tends to zero as $r\to\infty$. Thus $M_f(A)\ne0$ for all sufficiently
	large $|A|$. Moreover, the radius can be chosen uniformly for a family of functions satisfying that
	\[
	\norm f_{H^1(\Omega)}^2\leq C_0,
	\qquad
	-q_{\beta^2}^{\Omega}(f)\geq C_1>0,
	\]
	where $C_0>0$ and $C_1>0$ are fixed. Indeed, choose
	\[
	R\geq
	\max\left\{2D+2,\frac{2CC_0}{\gamma C_1}\right\},
	\]
	then
	\[
	M_f(A)\cdot\frac A{|A|}	\geq\gamma C_1-\frac{CC_0}{|A|}	\geq\frac{\gamma C_1}{2}>0,\quad |A|\geq R.
	\]
\end{remark}

We now consider the following equations simultaneously
\[
G(v,t)=0,\quad M_+(v,t,A_+)=0,\quad M_-(v,t,A_-)=0.
\]
The next lemma shows that all the solutions lie in one bounded set and that, for each fixed $v\in S(V)$, the Brouwer degree with respect to $(t,A_+,A_-)$ is $1$. In particular, at least one solution exists for each $v$.

\begin{lemma}[Solvability of the balance and center equations]\label{balance center solve degree}
	There exists a positive constant $R>0$ such that for every $(v,t)\in K$ and every $A\in\R^n$ with $|A|\geq R$,
	\begin{equation}
		M_\pm(v,t,A)\cdot\frac{A}{|A|}
		\geq\frac{\gamma\mu_*}{2}>0.
		\label{M pm far}
	\end{equation}
	Set
	\begin{equation}
		Y=(-S,S)\times B_R(0)\times B_R(0)\subset\R^{1+2n}
	\end{equation}
	and
	\begin{equation}
			\mathcal F(v,t,A_+,A_-)=\bigl(G(v,t),M_+(v,t,A_+),M_-(v,t,A_-)\bigr).
		\label{F joint map}
	\end{equation}
	The map $\mathcal F$ is continuous on $S(V)\times\R^{1+2n}$ and has no zeros on $S(V)\times\partial Y$. Besides, for every fixed $v\in S(V)$,
	\begin{equation}
		\deg\bigl(\mathcal F(v,\cdot),Y,0\bigr)=1.
		\label{degree formula}
	\end{equation}
	Moreover, the set of solutions
	\begin{equation}
		Z=\{(v,t,A_+,A_-)\in S(V)\times\R^{1+2n}:
		\mathcal F(v,t,A_+,A_-)=0\}
		\label{set of solutions}
	\end{equation}
	is a compact subset of $S(V)\times Y$, and for every $v\in S(V)$, there is at least one triple $(t,A_+,A_-)\in Y$ such that $(v,t,A_+,A_-)\in Z$.
\end{lemma}

\begin{proof}
\textbf{Step 1: A uniform bound on the centers.} Fix $(v,t)\in S(V)\times\R$. Since $f_\pm(v,t)\in H_0^1(\Omega)$, we may take $f=f_\pm(v,t)$ in the definition of $M_f$ from Remark \ref{general estimate on center}. For every $A\in\R^n$ and $i=1,\ldots,n$, it is straightforward to see that
\[
\begin{aligned}
	\bigl(M_{f_\pm(v,t)}(A)\bigr)_i=q_{\beta^2}^{\Omega}\bigl(F_{A,i}f_\pm(v,t),f_\pm(v,t)\bigr)=M_{\pm,i}(v,t,A).
\end{aligned}
\]
Hence
\[
M_\pm(v,t,A)=M_{f_\pm(v,t)}(A),
\quad
\forall\,(v,t,A)\in S(V)\times\R\times\R^n.
\]
Note that the maps $(v,t)\mapsto f_\pm(v,t)$ are continuous. Since $K$ is compact, there is a positive constant $C_0>0$ such that
	\[
	\norm{f_\pm(v,t)}_{H^1(\Omega)}^2\leq C_0,\quad \forall\,(v,t)\in K.
	\]
	By Lemma \ref{equal energy lemma}, we also have
	\[
	-q_{\beta^2}^{\Omega}(f_\pm(v,t))\geq\mu_*>0,\forall\,(v,t)\in K.
	\]
	Applying the estimate in Remark \ref{general estimate on center} with $C_1=\mu_*$ gives a radius
	$R>0$, independent of $(v,t)\in K$ and of the choice of the sign, such that
	\[
	M_\pm(v,t,A)\cdot\frac{A}{|A|}\geq\frac{\gamma\mu_*}{2}>0,\quad \forall |A|\geq R.
	\]

	\textbf{Step 2: Continuity and absence of boundary zeros.} The continuity of $\mathcal{F}$ is a direct corollary of Lemma \ref{basic lemma for sobolev function}, \eqref{continuity of e} and Lemma \ref{lemma of M and Q}. For $(v,t,A_+,A_-)\in S(V)\times\partial Y$, if $t=\pm S$, then $G(v,t)\ne0$ by	\eqref{threshold for G}, so $\mathcal F(v,t,A_+,A_-)\ne0$. If $|t|<S$, then at least one center has norm $R$. When $G(v,t)\ne0$, we know $\mathcal{F}\neq 0$. When $G(v,t)=0$, i.e., $(v,t)\in K$, \eqref{M pm far} shows that	$M_+(v,t,A_+)\ne0$ if $|A_+|=R$ and $M_-(v,t,A_-)\ne0$ if $|A_-|=R$. Therefore, $\mathcal F$ has no zeros on $S(V)\times\partial Y$.
	
	\textbf{Step 3: Computation of the degree.} Fix $v\in S(V)$. The map $\mathcal F(v,\cdot)$ is continuous on $\overline Y$ and has no zeros on $\partial Y$, so its Brouwer degree at zero is well defined. For basic knowledge of the Brouwer degree, we refer the reader to \cite{Deimling1985,Kesavan2022}. For $s\in[0,1]$, consider the continuous homotopy on $\overline Y$ given by
	\[
	(t,A_+,A_-)\mapsto
	\begin{pmatrix}
		G(v,t)\\
		(1-s)M_+(v,t,A_+)+sA_+\\
		(1-s)M_-(v,t,A_-)+sA_-
	\end{pmatrix}.
	\]
	Firstly, we prove that it has no zeros on $\partial Y$ for any $s\in[0,1]$. At $t=\pm S$, the first component is non-zero. If the first component $G(v,t)$ vanishes, then $(v,t)\in K$ and $|t|<S$, and hence the corresponding point on $\partial Y$ must satisfy $|A_+|=R$ or $|A_-|=R$. For $\sigma\in\{+,-\}$ with $|A_\sigma|=R$, \eqref{M pm far} gives
	\[
	\begin{aligned}
		&\bigl((1-s)M_\sigma(v,t,A_\sigma)+sA_\sigma\bigr)\cdot\frac{A_\sigma}{R}\geq(1-s)\frac{\gamma\mu_*}{2}+sR>0.
	\end{aligned}
	\]
	Therefore, by the homotopy invariance of degree, we have
	$$
	\deg\bigl(\mathcal F(v,\cdot),Y,0\bigr)=\deg\bigl((G(v,t),A_+,A_-),Y,0\bigr).
	$$
	Next, for $s\in[0,1]$, consider the following another continuous homotopy on $\overline Y$ given by
	\[
	(t,A_+,A_-)\mapsto\bigl((1-s)G(v,t)+st,A_+,A_-\bigr).
	\]
	At $t=S$, the first component is strictly positive because $G(v,S)>0$ and $S>0$. At $t=-S$, it is strictly negative because $G(v,-S)<0$ and $-S<0$. If $|A_+|=R$ or $|A_-|=R$, the corresponding center component is non-zero. Therefore, for any $s\in[0,1]$ the homotopy given above has no zeros on $\partial Y$. When $s=1$, it is the identity map on $\overline Y$. Since $0\in Y$, the homotopy invariance of degree shows that
	\[
	\deg\bigl(\mathcal F(v,\cdot),Y,0\bigr)=\deg(\mathrm{Id},Y,0)=1.
	\]
	
	\textbf{Step 4: Compactness of $Z$ and existence of solutions for each $v$.} Let $(v,t,A_+,A_-)\in Z$. Then $G(v,t)=0$, so that $(v,t)\in K$ and $|t|<S$. On the other hand, we know $M_\pm(v,t,A_\pm)=0$. By \eqref{M pm far}, we have $|A_+|<R$ and $|A_-|<R$, which proves that $Z\subset S(V)\times Y$. Consequently, $Z$ can be viewed as the zero set of the restriction of $\mathcal F$ to
	\[
	S(V)\times[-S,S]\times
	\overline{B_R(0)}\times\overline{B_R(0)},
	\]
	which is a compact set. By the continuity of $\mathcal{F}$, we know $Z$ is compact. Finally, for each fixed $v\in S(V)$, by \eqref{degree formula}, we know $\mathcal F(v,\cdot)$ has a zero in $Y$. The proof is complete.
\end{proof}

Notice that we do not select a specific triple $(t,A_+,A_-)$ for each $v$, so neither uniqueness nor a continuous choice of these parameters is required.

We now distinguish what holds on $Z$ from what can be concluded for nearby parameters.

For $\zeta=(v,t,A_+,A_-)\in S(V)\times Y$, write
\[
	e=e(v,t),\quad f_\pm=f_\pm(v,t),\quad M_\pm=M_\pm(v,t,A_\pm),\quad Q_\pm=Q_\pm(v,t,A_\pm)
\]
for simplicity. If $\zeta\in Z$, then $G(v,t)=0$ and $M_+=M_-=0$. Then by the definition of $M_\pm$, we know
\[
q_{\beta^2}^{\Omega}(F_{A_\pm,i}f_\pm,f_\pm)=0,\quad i=1,\ldots,n.
\]
The equation $G(v,t)=0$ gives $(v,t)\in K$, so $q_{\alpha^2}^{\Omega}(f_\pm)\leq0$. By Proposition \ref{energy identity prop}, we obtain that
\begin{equation}
	\tr Q_\pm=\sum_{i=1}^nq_{\beta^2}^{\Omega}(F_{A_\pm,i}f_\pm)\leq\gamma^2q_{\alpha^2}^{\Omega}(f_\pm)\leq0.
	\label{trace is negative in Z}
\end{equation}
We only use the fact that $(v,t)\in K$ to establish the trace estimate \eqref{trace is negative in Z}.

For general $\zeta\in S(V)\times Y$, the balance and center equations no longer hold. For $\sigma\in\{+,-\}$, $c\in\R^n$, and $\rho\in\R$, it is easy to verify that
\[
	q_{\beta^2}^{\Omega}\left(\sum_{i=1}^n c_iF_{A_\sigma,i}f_\sigma+\rho f_\sigma\right)=c^{\mathsf T}Q_\sigma c+2\rho\,c\cdot M_\sigma+\rho^2q_{\beta^2}^{\Omega}(f_\sigma).
\]
The mixed term $2\rho\,c\cdot M_\sigma$ vanishes on $Z$. For general $\zeta\in S(V)\times Y$, however, it need not vanish and hence can not be dropped. Similarly, the definition of $G$ and \eqref{energy split} give
\[
	q_{\alpha^2}^{\Omega}(f_+)=\frac12\bigl(q_{\alpha^2}^{\Omega}(e)+G(v,t)\bigr),\quad	q_{\alpha^2}^{\Omega}(f_-)=\frac12\bigl(q_{\alpha^2}^{\Omega}(e)-G(v,t)\bigr).
\]
Since $q_{\alpha^2}^{\Omega}(e)\leq0$, we know $q_{\alpha^2}^\Omega(f_\pm)\leq |G(v,t)|/2$. Therefore, by Proposition \ref{energy identity prop}, we know
\[
\tr Q_\pm=\sum_{i=1}^nq_{\beta^2}^{\Omega}(F_{A_\pm,i}f_\pm)\leq\gamma^2q_{\alpha^2}^{\Omega}(f_\pm)\leq\frac{\gamma^2}{2}|G(v,t)|.
\]
When $G(v,t)=0$, this inequality is precisely \eqref{trace is negative in Z}. When $G(v,t)\ne0$, the right-hand side is positive, and hence the estimate above does not give $\tr Q_\pm\leq0$.

\section{A degree argument and the proof of Theorem \ref{main result}}\label{section5}

Recall that, after a domain dilation, we have normalized $\lambda_4(\Omega)=\alpha^2$. We have continued to assume, for a contradiction, that $\lambda_5(\Omega)>\beta^2$. In this section, we show that the two symmetric matrices $Q_+$ and $Q_-$ defined below each have exactly one negative eigenvalue in a neighborhood of the solution set of $Z$. In this neighborhood, we construct a continuous family of linear maps that send every non-zero vector in the negative eigenspace of $Q_+$ to a non-zero vector perpendicular to the corresponding $v\in S(V)$. We then show that these properties are in  contradiction with the degree calculation in Lemma \ref{balance center solve degree}.

By Lemma \ref{balance center solve degree}, the map defined in \eqref{F joint map},
\[
\mathcal F=(G,M_+,M_-):
S(V)\times\R^{1+2n}\longrightarrow\R^{1+2n},\quad Y=(-S,S)\times B_R(0)\times B_R(0),
\]
is continuous. The zero set $Z=\mathcal F^{-1}(0)$ is compact and contained in $S(V)\times Y$. For fixed $v\in S(V)$, the constraint map $\mathcal F(v,\cdot)$ has no zero on $\partial Y$ and has degree $1$ in $Y$ at zero. For $\zeta=(v,t,A_+,A_-)\in S(V)\times Y$, recall from \eqref{definition of e original} that
\[
e=e(v,t)=v-tu_1,\quad f_+=e^+,\quad f_-=e^-.
\]
Define
\begin{equation}
	\psi_i^+\coloneqq F_{A_+,i}f_+,\quad\psi_i^-\coloneqq F_{A_-,i}f_-,\quad i=1,\cdots, n,
	\label{trial function psi}
\end{equation}
and the real symmetric matrices
\begin{equation}
	Q_\pm(\zeta)=\bigl(q_{\beta^2}^{\Omega}(\psi_i^\pm,\psi_j^\pm)\bigr)_{i,j=1}^n.
	\label{Qpm sec 5}
\end{equation}
These definitions are valid throughout $S(V)\times Y$, and by Lemma \ref{lemma of M and Q}, we know that $Q_\pm$ is continuous in all parameters. The equations defining $Z$ will firstly be used to determine the eigenvalues of $Q_\pm$ on $Z$.

\begin{proposition}\label{neighborhood of Z}
	There exists an open neighborhood $\mathcal O$ of $Z$ in $S(V)\times Y$ such that each of $Q_+(\zeta)$ and $Q_-(\zeta)$	has exactly one negative eigenvalue and $n-1$ positive eigenvalues for every $\zeta\in\mathcal O$.
	
	Let $\Pi(\zeta)$ be the orthogonal projection onto the negative eigenspace of $Q_+(\zeta)$. Then $\Pi$ is continuous in $\mathcal O$ and has rank one. Furthermore, there exists a continuous map $T:\mathcal O\to\R^{3\times n}$ such that
	\begin{equation}
		v\cdot T(\zeta)c=0,\quad  \forall\, c\in\R^n,	\quad T(\zeta)c\ne0,\quad	\forall\, c\in\operatorname{range}\Pi(\zeta)\setminus\{0\},
		\label{orthogonal condition Pi}
	\end{equation}
	for every $\zeta=(v,t,A_+,A_-)\in\mathcal O$.
\end{proposition}

\begin{proof}
	\textbf{Step 1: Determining the eigenvalues of $Q_+$ and $Q_-$ on $Z$.}	Let $$E_4=\spanop\{u_1,u_2,u_3,u_4\}$$ and let $P_4$ be the
	$L^2$-orthogonal projection onto $E_4$. Since $\lambda_4(\Omega)=\alpha^2<\beta^2<\lambda_5(\Omega)$, we obtain by Lemma \ref{q-energy definite subspace} that
	\begin{equation}
		q_{\beta^2}^{\Omega}(h)
		\ge(\lambda_5-\beta^2)\norm h_{L^2(\Omega)}^2>0
		\quad\text{for }0\ne h\in H_0^1(\Omega)\cap E_4^\perp.
		\label{energy estimate in ortho complement}
	\end{equation}
	Consequently, any subspace on which $q_{\beta^2}^{\Omega}$ is non-positive has dimension at most four. 
	
	Fix $\zeta\in Z$. By the definition of $\mathcal{F}$, we know $G(v,t)=0$, so $(v,t)\in K$. Lemma \ref{equal energy lemma} implies that both $f_+$
	and $f_-$ are non-zero. Moreover, $q_{\alpha^2}^{\Omega}(f_+)=q_{\alpha^2}^{\Omega}(f_-)\leq0$. Since $d=\beta^2-\alpha^2>0$, we have
	\begin{equation}
		q_{\beta^2}^{\Omega}(f_\pm)
		=q_{\alpha^2}^{\Omega}(f_\pm)
		-d\norm{f_\pm}_{L^2(\Omega)}^2<0.
		\label{top:baseenergy}
	\end{equation}
	Since $f_+f_-=0$ almost everywhere,	we know $f_+$ and $f_-$ are linearly independent. Moreover, by Lemma~\ref{basic lemma for sobolev function}, we know
	$q_{\beta^2}^{\Omega}(f_+,f_-)=0$. Therefore, it is easy to deduce that
	\[
	N=\spanop\{f_+,f_-\}
	\]
	has dimension two and $q_{\beta^2}^{\Omega}$ is negative definite on $N$. On the other hand, since $\zeta\in Z$, we have $M_\pm(v,t,A_\pm)=0$. By \eqref{moment M}, we have
	\begin{equation}
		q_{\beta^2}^{\Omega}(\psi_i^+,f_+)=0,\quad	q_{\beta^2}^{\Omega}(\psi_i^-,f_-)=0,\quad i=1,\cdots,n.
		\label{zhengjiao 1}
	\end{equation}
	By Lemma \ref{basic lemma for sobolev function}, we know
	\begin{equation}
		q_{\beta^2}^{\Omega}(\psi_i^+,f_-)=q_{\beta^2}^{\Omega}(f_+,\psi_j^-)=q_{\beta^2}^{\Omega}(\psi_i^+,\psi_j^-)=0,\quad 1\le i,j\le n.\label{zhengjiao 2}
	\end{equation}
	Set
	\[
	W_\pm=\spanop\{\psi_1^\pm,\ldots,\psi_n^\pm\},
	\]
	then we know by \eqref{zhengjiao 1} and \eqref{zhengjiao 2} that $N$, $W_+$, and $W_-$ are pairwise $q_{\beta^2}^{\Omega}$-orthogonal. Besides, each of $W_+$ and $W_-$ has dimension $n$. To see this, take $0\ne c\in\R^n$. Away from $x=A_\pm$, we have
	\[
	\sum_{i=1}^n c_i\psi_i^\pm(x)=\frac{g(|x-A_\pm|)}{|x-A_\pm|}\bigl(c\cdot(x-A_\pm)\bigr)f_\pm(x).
	\]
	The radial factor $g(|x-A_\pm|)/|x-A_\pm|$ is strictly positive and the linear factor $c\cdot(x-A_\pm)$ vanishes only on a hyperplane, which has Lebesgue measure zero. However, we know $f_\pm$ is non-zero on a set of positive measure. These facts imply that $\dim W_\pm=n$. Obviously, by the definition of $\psi_i^{\pm}$, we have $W_+\cap W_-=\{0\}$. Also, if $h\in N\cap(W_++W_-)$, it is easy to prove that $q_{\beta^2}^{\Omega}(h)=0$, and the negative definiteness of $q_{\beta^2}^\Omega$ on $N$ implies $h=0$. The above analysis actually shows that
	\begin{equation}
		N\mathbin{\oplus}W_+\mathbin{\oplus}W_-\quad \text{is a direct sum of pairwise $q_{\beta^2}^\Omega$-orthogonal spaces.}
		\label{direct sum N W+ W-}
	\end{equation}
	Applying \eqref{energy identity omega} separately to $f_+$ and $f_-$ gives
	\begin{equation}
		\tr Q_\pm
		=\sum_{i=1}^n q_{\beta^2}^{\Omega}(\psi_i^\pm)
		\le\gamma^2q_{\alpha^2}^{\Omega}(f_\pm)\le0.
		\label{trace bound sec 5}
	\end{equation}
	Let $r_\pm$ be the number of non-positive eigenvalues of $Q_\pm$. The trace bound \eqref{trace bound sec 5} shows that $r_\pm\ge1$. Let $C_\pm\subset\R^n$ be the subspaces spanned by the eigenvectors corresponding to non-positive eigenvalues. The maps
	\[
	C_\pm\rightarrow W_\pm,\quad
	c\mapsto\sum_{i=1}^n c_i\psi_i^\pm
	\]
	are injective, and both the images have dimension $r_\pm$. Note that
	\[
	q_{\beta^2}^{\Omega}\left(\sum_{i=1}^n c_i\psi_i^\pm\right)
	=c^{\mathsf T}Q_\pm c\le0,\quad\forall\,c\in C_\pm.
	\]
	Hence, $N\oplus W_+\oplus W_-$ contains a linear subspace of dimension $2+r_++r_-$ on which $q_{\beta^2}^{\Omega}$ is nonpositive. Therefore, we have
	\[
	2+r_++r_-\le4.
	\]
	It then follows that $r_+=r_-=1$.
	
	For a real symmetric matrix $Q$, let $\theta_1(Q)\le\cdots\le\theta_n(Q)$ denote its eigenvalues. We have shown that
	$$\theta_1(Q_\pm)\le0<\theta_2(Q_\pm)\le\cdots\le\theta_n(Q_\pm).$$
	Since $n\ge2$, and $\tr Q_\pm\leq0$, we know
	\[
	\theta_1(Q_\pm)\le-\sum_{j=2}^n\theta_j(Q_\pm)<0.
	\]
	Thus each of $Q_\pm$ has exactly one negative eigenvalue and $n-1$ positive eigenvalues on $Z$.
	
	\textbf{Step 2: Construction of the map $T$ on $Z$.} For $\zeta\in Z$, set
	\[
	\ell_\pm=\ip{f_\pm}{u_1}>0,\quad m_\pm=\norm{f_\pm}_{L^2(\Omega)}^2>0,
	\]
	where we have used Lemma \ref{equal energy lemma}. Now, we choose constants $h_+$ and $h_-$ such that $\psi+h_+f_++h_-f_-$ is $L^2$-orthogonal to both $u_1$ and $v$ for some given function $\psi$. A priori calculation indicates that
	\[
	D(\zeta)\begin{pmatrix}h_+\\h_-\end{pmatrix}
	=-\begin{pmatrix}\ip{\psi}{u_1}\\\ip{\psi}{v}\end{pmatrix},
	\]
	where
	\begin{equation}
		D(\zeta)=
		\begin{pmatrix}
			\ip{f_+}{u_1}&\ip{f_-}{u_1}\\
			\ip{f_+}{v}&\ip{f_-}{v}
		\end{pmatrix}
		=\begin{pmatrix}
			\ell_+&\ell_-\\
			m_++t\ell_+&-m_-+t\ell_-
		\end{pmatrix}.
		\label{D zeta}
	\end{equation}
	In particular,
	\begin{equation}
		\det D(\zeta)=-\ell_+m_--\ell_-m_+<0.
		\label{D determinant}
	\end{equation}
	For $c\in\R^n$, define
	\[
	\psi_c\coloneqq\sum_{i=1}^n c_i\psi_i^+,
	\quad
	\begin{pmatrix}h_+\\h_-\end{pmatrix}
	\coloneqq -D(\zeta)^{-1}
	\begin{pmatrix}\ip{\psi_c}{u_1}\\\ip{\psi_c}{v}\end{pmatrix},
	\]
	and 
	\begin{equation}
		z_\zeta(c)\coloneqq\psi_c+h_+f_++h_-f_-.
		\label{top:correctedtrial}
	\end{equation}
	The function $z_\zeta$ is linear in $c$ and satisfies
	\begin{equation}
		\ip{z_\zeta(c)}{u_1}=0,\quad
		\ip{z_\zeta(c)}{v}=0.
		\label{z-zeta orthogonal}
	\end{equation}

	Let $L_\zeta\subset\R^n$ be the one-dimensional negative
	eigenspace of $Q_+(\zeta)$. For $0\ne c\in L_\zeta$, by the orthogonal properties established in
	Step 1, we have
	\begin{align}
		q_{\beta^2}^{\Omega}\bigl(z_\zeta(c)\bigr)=c^{\mathsf T}Q_+(\zeta)c+h_+^2q_{\beta^2}^{\Omega}(f_+)+h_-^2q_{\beta^2}^{\Omega}(f_-)\le\theta_1(Q_+(\zeta))|c|^2<0.
		\label{z-zeta is negative}
	\end{align}
	Define $T(\zeta):\R^n\to\R^3$ by
	\begin{equation}
		T(\zeta)c\coloneqq
		\begin{pmatrix}
			\ip{z_\zeta(c)}{u_2}\\
			\ip{z_\zeta(c)}{u_3}\\
			\ip{z_\zeta(c)}{u_4}
		\end{pmatrix}.
		\label{def of map T}
	\end{equation}
	Note that this is just the coordinate vector of the $L^2$-orthogonal projection of $z_\zeta(c)$ onto $V$. For $v=\sum_{j=2}^4 a_ju_j$, \eqref{z-zeta orthogonal} gives
	\[
	v\cdot T(\zeta)c
	=\sum_{j=2}^4 a_j\ip{z_\zeta(c)}{u_j}
	=\ip{z_\zeta(c)}{v}=0.
	\]
	If $T(\zeta)c=0$ for some $0\ne c\in L_\zeta$, then $z_\zeta(c)\perp V$ and $z_\zeta(c)\perp u_1$, so $z_\zeta(c)\in E_4^\perp$. This is a contradiction by \eqref{energy estimate in ortho complement} and \eqref{z-zeta is negative}. Hence, \eqref{orthogonal condition Pi} holds on $Z$.
	
	\textbf{Step 3: Spectral bounds and continuity of $\Pi$.} Notice that $Q_\pm$ and $D$ are continuous in $\zeta$. The ordered eigenvalues are obviously continuous as well. By the compactness of $Z$, \eqref{D determinant}, and results in Step 1, we know there exists a constant $\varepsilon_0>0$ such that, on $Z$,
	\begin{equation}
		|\det D|\geq\varepsilon_0,\quad\theta_1(Q_\pm)\leq-\varepsilon_0,\quad	\theta_2(Q_\pm)\geq\varepsilon_0.
		\label{continuity start}
	\end{equation}
	By continuity, there is an open neighborhood $\mathcal O_0$ of
	$Z$ in $S(V)\times Y$ such that
	\[
	|\det D|>\frac{\varepsilon_0}{2},\quad\theta_1(Q_\pm)<-\frac{\varepsilon_0}{2},\quad\theta_2(Q_\pm)>\frac{\varepsilon_0}{2},\quad \forall\,\zeta\in \mathcal{O}_0.
	\]
	Thus $D$ is invertible and both $Q_+$ and $Q_-$ have exactly one negative eigenvalue and $n-1$ positive eigenvalues in $\mathcal O_0$.
	
	For $\zeta\in\mathcal O_0$, let $\Pi(\zeta)$ be the orthogonal projection onto the negative eigenspace of $Q_+(\zeta)$. Choose a unit eigenvector $c\in\R^n$ corresponding to the unique negative eigenvalue of $Q_+(\zeta)$. Hence
	\[
	|c|=1,\quad	Q_+(\zeta)c=\theta_1(Q_+(\zeta))c,
	\]
	and the orthogonal projection is given by
	\[
	\Pi(\zeta)=cc^{\mathsf T},\quad \Pi(\zeta)b=(b\cdot c)c,\quad b\in\R^n.
	\]
	The projection matrix is unchanged when $c$ is replaced by $-c$. Also notice that this choice is made at each fixed $\zeta$. Take $\zeta_j\to\zeta$ in $\mathcal O_0$ and choose unit eigenvector $c_j$ corresponding to the unique negative eigenvalue of $Q_+(\zeta_j)$. Up to a subsequence, we may assume that $c_j\to c$ with $|c|=1$. Since
	\[
	Q_+(\zeta_j)c_j=\theta_1(Q_+(\zeta_j))c_j,
	\]
	by the continuity of $\eta_1$, we know $c$ spans the negative eigenspace of $Q_+(\zeta)$ and
	$$
	\Pi(\zeta_j)=c_jc_j^{\mathsf T}\to cc^{\mathsf T}=\Pi(\zeta).
	$$
	Hence the projection $\Pi$ is continuous.
	
	We next construct a continuous unit negative eigenvector on a neighborhood of each fixed point $\zeta_0\in\mathcal O_0$. Choose
	$$
	c_0\in\operatorname{range}\Pi(\zeta_0),
	\quad |c_0|=1.
	$$
	Since $\Pi(\zeta_0)$ is the orthogonal projection onto this space, $\Pi(\zeta_0)c_0=c_0$. Since $\Pi$ is continuous, there exists an open neighborhood $U_{\zeta_0}\subset\mathcal{O}_0$ of $\zeta_0$ such that
	\[
	|\Pi(\zeta)c_0-c_0|<\frac12, \quad\zeta\in\mathcal U_{\zeta_0}.
	\]
	Therefore,
	$$
	|\Pi(\zeta)c_0|\geq |c_0|-|\Pi(\zeta)c_0-c_0|\geq\frac{1}{2}
	$$
	and we may define
	\begin{equation}
		c(\zeta)=\frac{\Pi(\zeta)c_0}{|\Pi(\zeta)c_0|},\quad\zeta\in\mathcal U_{\zeta_0}.
		\label{local continuous eigenvector}
	\end{equation}
	By the definition, we know $|c(\zeta)|=1$ and $c(\zeta)\in\operatorname{range}\Pi(\zeta)$. Since $\operatorname{range}\Pi(\zeta)$ is precisely the negative eigenspace of $Q_+(\zeta)$, we have
	\[
	Q_+(\zeta)c(\zeta)=\theta_1(Q_+(\zeta))c(\zeta).
	\]
	Moreover, $c:\mathcal U_{\zeta_0}\to\R^n$ is continuous, with $c(\zeta_0)=c_0$.
	
	\textbf{Step 4: Extension of $T$.} In the neighborhood $\mathcal{O}_0$, we may also use \eqref{D zeta}, \eqref{top:correctedtrial}, and \eqref{def of map T} to define $z_\zeta(c)$ and $T(\zeta)$. These formulas are well-defined because $D(\zeta)$ is invertible for all $\zeta\in\mathcal{O}_0$. By the continuity property established before, we know that the map $(\zeta,c)\mapsto z_\zeta(c)$ is continuous into $H_0^1(\Omega)$. For each fixed $\zeta$, all these formulas are linear in $c$ and it is not hard to verify that $T$ is continuous in $\mathcal O_0$. By the definition of $z_\zeta$, we also have
	\begin{equation}
		\ip{z_\zeta(c)}{u_1}=0,\quad\ip{z_\zeta(c)}{v}=0,\quad\forall\,\zeta\in\mathcal O_0\label{orthogonal z-zeta and u1 v neighborhood O0}
	\end{equation}
	for all $c\in\Rn$. In particular, in $\mathcal O_0$, we have
	\[
	v\cdot T(\zeta)c=\ip{z_\zeta(c)}{v}=0.
	\]

	It remains to prove that $T(\zeta)c$ is non-zero when $c$ is a non-zero vector in $\operatorname{range}\Pi(\zeta)$, after restricting to a smaller neighborhood than $\mathcal O_0$. For each $\zeta\in\mathcal O_0$, choose
	$c\in\operatorname{range}\Pi(\zeta)$ with $|c|=1$
	and set
	\[
	E(\zeta)=q_{\beta^2}^{\Omega}\bigl(z_\zeta(c)\bigr).
	\]
	This definition does not depend on the choice of the unit vector $c$. Indeed, the only other choice is $-c$, and the linearity of $z_\zeta$ gives
	\[
	q_{\beta^2}^{\Omega}\bigl(z_\zeta(-c)\bigr)=q_{\beta^2}^{\Omega}\bigl(-z_\zeta(c)\bigr)=q_{\beta^2}^{\Omega}\bigl(z_\zeta(c)\bigr).
	\]
	To establish the continuity of $E$, fix $\zeta_0\in\mathcal O_0$ and use \eqref{local continuous eigenvector} to choose a continuous unit negative eigenvector $c(\zeta)$ in a neighborhood of $\zeta_0$. In this neighborhood, we know
	\[
	E(\zeta)
	=q_{\beta^2}^{\Omega}\bigl(z_\zeta(c(\zeta))\bigr),
	\]
	and we immediately know $E$ is continuous near $\zeta_0$ since $\zeta\mapsto z_\zeta(c(\zeta))$ and $q_{\beta^2}^{\Omega}$ are continuous. Since $\zeta_0$ was arbitrary chosen, we know $E$ is continuous in $\mathcal O_0$.
	
	For $\zeta\in Z$, from \eqref{z-zeta is negative} and \eqref{continuity start}, we know
	\[
	E(\zeta)\le\theta_1(Q_+(\zeta))\le-\varepsilon_0.
	\]
	Define
	\[
	\mathcal O\coloneqq\left\{\zeta\in\mathcal O_0:E(\zeta)<-\frac{\varepsilon_0}{2}\right\},
	\]
	which is obviously an open set and contains $Z$. If $\zeta\in\mathcal O$ and
	$0\ne c\in\operatorname{range}\Pi(\zeta)$,
	then $c/|c|$ is a unit vector in $\operatorname{range}\Pi(\zeta)$. The linearity of $z_\zeta$ and homogeneity of the quadratic-form then give
	\begin{equation}
			q_{\beta^2}^{\Omega}\bigl(z_\zeta(c)\bigr)=|c|^2q_{\beta^2}^{\Omega}\bigl(z_\zeta(c/|c|)\bigr)=|c|^2E(\zeta)<-\frac{\varepsilon_0}{2}|c|^2.
	\label{z zeta neighborhood q beta}
	\end{equation}
	If $T(\zeta)c=0$ for some $\zeta\in\mathcal O$ and $0\ne c\in\operatorname{range}\Pi(\zeta)$, then by the definition of $T$, we know $z_\zeta(c)$ is orthogonal to $u_2,u_3,u_4$. Together with \eqref{orthogonal z-zeta and u1 v neighborhood O0}, we have $z_\zeta(c)\in E_4^\perp$. By \eqref{energy estimate in ortho complement}, we have
	\[
	q_{\beta^2}^{\Omega}\bigl(z_\zeta(c)\bigr)\geq(\lambda_5-\beta^2)\norm{z_\zeta(c)}_{L^2(\Omega)}^2\geq0,
	\]
	contradicting \eqref{z zeta neighborhood q beta}. The proof is complete.
\end{proof}

We next show that the matrices constructed in Proposition \ref{neighborhood of Z} are incompatible with the degree property of the balance and center equations $\mathcal{F}=0$.
\begin{lemma}\label{contradiction obstruction}
	Let $Y\subset\R^m$ be a bounded open set and $U\subset\R^m$ be an open neighborhood of $\overline Y$. Let
	\[
	\mathcal F:\mathbb S^2\times U\longrightarrow\R^m
	\]
	be a continuous map. Suppose that
	$$
	\mathcal F(v,y)\ne0\quad\text{for all }v\in\mathbb S^2,\ y\in\partial Y,
	$$
	and
	$$
		\deg(\mathcal F(v,\cdot),Y,0)=1\quad\text{for all }v\in\mathbb S^2.
	$$
	Define
	\[
	K_0\coloneqq\{(v,y)\in\mathbb S^2\times\overline Y:\mathcal F(v,y)=0\}.
	\]
	
	Let $r\ge1$ be an integer, and let $\mathcal O$ be any open neighborhood of $K_0$ in $\mathbb S^2\times U$. Then there are no continuous maps
	\[
	\Pi:\mathcal O\longrightarrow\R^{r\times r},\quad T:\mathcal O\longrightarrow\R^{3\times r}
	\]
	satisfying both of the following conditions at every $\zeta=(v,y)\in\mathcal O$:
	\begin{itemize}
		\item[1.] $\Pi(\zeta)$ is an orthogonal projection of rank one:
		\[
		\Pi(\zeta)^{\mathsf T}=\Pi(\zeta),\quad	\Pi(\zeta)^2=\Pi(\zeta),\quad\operatorname{rank}\Pi(\zeta)=1.
		\]
		
		\item[2.] $T(\zeta)$ maps $\R^r$ into $v^\perp$ and maps every non-zero vector in $\operatorname{range}\Pi(\zeta)$ to a non-zero vector:
		\begin{equation}
			v\cdot T(\zeta)c=0,\quad c\in\R^r,\quad\text{and}\quad T(\zeta)c\ne0,
			\quad 0\ne c\in\operatorname{range}\Pi(\zeta).
			\label{property of T}
		\end{equation}
	\end{itemize}
\end{lemma}

\begin{proof}
	We proceed by contradiction. Assume that, for some integer $r\ge1$, there exist continuous maps $\Pi$ and $T$, defined on an open neighborhood $\mathcal O$ of $K_0$, that satisfy the stated conditions.
	
	\textbf{Step 1: A continuous projection in $\R^3$.} We first construct the orthogonal projection onto $T(v,y)(\operatorname{range}\Pi(v,y))\subset v^\perp$. Fix $(v,y)\in\mathcal O$ and choose a unit vector $q\in\R^r$ such that $\spanop \{q\}=\operatorname{range}\Pi(v,y)$. Since $\Pi(v,y)$ is an orthogonal projection of rank one, we may write
	\[
	\Pi(v,y)=qq^{\mathsf T}.
	\]
	Moreover, by \eqref{property of T}, we know $T(v,y)q\ne0$ and $v\cdot T(v,y)q=0$. Therefore, we have
	$$
	T(v,y)\Pi(v,y)T(v,y)^{\mathsf T}=(T(v,y)q)(T(v,y)q)^{\mathsf T}
	$$
	and
	$$
	\tr\bigl(T(v,y)\Pi(v,y)T(v,y)^{\mathsf T}\bigr)=|T(v,y)q|^2>0.
	$$
	We may thus define
	\begin{equation}
		P(v,y)\coloneqq
		\frac{T(v,y)\Pi(v,y)T(v,y)^{\mathsf T}}
		{\tr\bigl(T(v,y)\Pi(v,y)T(v,y)^{\mathsf T}\bigr)}.
		\label{projection lemma5.2}
	\end{equation}
	At $(v,y)$, we may write
	$$
	P(v,y)=aa^{\mathsf T},\quad a=\frac{T(v,y)q}{|T(v,y)q|}.
	$$
	Since $|a|=1$ and $a\cdot v=0$, the matrix $P(v,y)$ is the orthogonal projection onto $\spanop \{a\}$. In particular,
	\begin{equation}
		P^{\mathsf T}=P,\quad P^2=P,\quad\operatorname{rank}P=1,\quad Pv=0.
		\label{property of P}
	\end{equation}
	Since $\Pi$ and $T$ are continuous and the denominator in \eqref{projection lemma5.2} is positive, we know $P$ is continuous in $\mathcal O$.

	Next, we choose a neighborhood $W$ of $K_0$ whose closure is compact and contained in $\mathcal O\cap(\mathbb S^2\times Y)$. By the continuity of $\mathcal F$, we know $K_0$ is compact. Since $\mathcal{F}(v,y)\neq0$ for all $(v,y)\in \mathbb S^2\times\partial Y$, there holds
	$$
	K_0\subset\mathcal O\cap(\mathbb S^2\times Y).
	$$
	Since $\mathcal O\cap(\mathbb S^2\times Y)$ is relatively open in $\mathbb S^2\times U$, around each point of $K_0$, we can choose a sufficiently small relative open ball whose closure is compact and contained in $\mathcal O\cap(\mathbb S^2\times Y)$. By the compactness of $K_0$, we may choose finitely many such balls that cover $K_0$. Let $W$ be their union, then $W$ is relatively open in $\mathbb S^2\times U$ and satisfies
	\begin{equation}
			K_0\subset W,\quad\overline W\subset\mathcal O\cap(\mathbb S^2\times Y),\quad\overline W\text{ is compact}.
		\label{the new nbhd W}
	\end{equation}
	Since $\overline W\subset\mathcal O$, the projection $P$ defined above is continuous on $\overline W$.
	
	\textbf{Step 2: An auxiliary degree calculation.} Denote
	\[
	B_\rho=\{c\in\R^3:|c|<\rho\},\quad\mathcal A=B_2\setminus\overline{B_{1/2}}.
	\]
	For a rank-one orthogonal projection $P$ in $\R^3$, define
	\begin{equation}
		E_P(c)\coloneqq (I_3-P)c+(|c|^2-1)Pc.
		\label{Ep definition}
	\end{equation}
	The two terms on the right-hand side are orthogonal, so
	\[
	|E_P(c)|^2=|(I_3-P)c|^2+(|c|^2-1)^2|Pc|^2.
	\]
	For $c\in\overline{\mathcal A}$, we know $c\neq0$ and hence
	\begin{equation}
		E_P(c)=0\quad\Longleftrightarrow\quad Pc=c\,\,\text{ and }\,\,|c|=1.
		\label{Ep zeros}
	\end{equation}
	Thus the zeros of $E_P$ in $\overline{\mathcal A}$ are exactly the two unit vectors in $\operatorname{range}P$, which shows that $E_P$ has no zeros on $\partial\mathcal{A}$. To compute the degree, we determine the contribution of each zero.
	
	Let $c_*$ be a zero of $E_P$. Since $Pc_*=c_*$ and $|c_*|=1$, by differentiating \eqref{Ep definition}, we see that for every $h\in\R^3$,
	\[
	\begin{aligned}
		DE_P(c_*)h&=(I_3-P)h+2(c_*\cdot h)Pc_*+(|c_*|^2-1)Ph\\
		&=(I_3-P)h+2(c_*\cdot h)c_*.
	\end{aligned}
	\]
	Since $\spanop\{c_*\}=\operatorname{range}P$, we may write $P=c_*c_*^{\mathsf T}$. Therefore,
	\[
	DE_P(c_*)h=(I_3-P)h+2Ph=(I_3+P)h.
	\]
	The matrix $I_3+P$ acts as multiplication by $2$ on $\operatorname{range}P$ and as the identity map on the two-dimensional orthogonal complement of $\operatorname{range}P$. Hence
	\[
	\det DE_P(c_*)=2>0.
	\]
	Both zeros are therefore non-degenerate and each has local degree $+1$. Since in $\overline{\mathcal A}$, there are no other zeros, we have
	\[
	\deg(E_P,\mathcal A,0)=1+1=2.
	\]

	We next consider the case where the projection $P$ depends
	continuously on a parameter $y$. Let $\Theta\subset\R^m$ be a
	bounded open set, and let
	\[
	f:\overline{\Theta}\longrightarrow\R^m
	\]
	be a continuous map satisfying
	\[
	f(y)\ne0
	\qquad\text{for every }y\in\partial\Theta.
	\]
	Suppose that
	\[
	P:\overline{\Theta}\longrightarrow\R^{3\times3}
	\]
	is continuous and that, for every $y\in\overline{\Theta}$,
	the matrix $P(y)$ is a rank-one orthogonal projection. Define
	\[
	d_f\coloneqq\deg(f,\Theta,0),
	\qquad
	\mathcal G(y,c)
	\coloneqq
	\bigl(f(y),E_{P(y)}(c)\bigr),
	\qquad
	(y,c)\in\overline{\Theta}\times\overline{B_2}.
	\]
	The map $\mathcal G$ is continuous. For
	$(y,c)\in\partial(\Theta\times\mathcal A)$, either
	$y\in\partial\Theta$ or $|c|\in\{1/2,2\}$. In the first case,
	$f(y)\ne0$, while in the second case,
	$E_{P(y)}(c)\ne0$ by~\eqref{Ep zeros}. Thus $\mathcal G$ has
	no zeros on $\partial(\Theta\times\mathcal A)$, and its
	Brouwer degree on $\Theta\times\mathcal A$ is well defined.
	
	We claim that
	\begin{equation}
		\deg(\mathcal G,\Theta\times\mathcal A,0)=2d_f.
		\label{degree of G on annulus}
	\end{equation}
	We first compute the degrees on $\Theta\times B_2$ and
	$\Theta\times B_{1/2}$, and then use the additivity property.
	
	We begin with $\Theta\times B_2$. For $0\le t\le1$, define
	\[
	\mathcal G_t^+(y,c)\coloneqq
	\begin{pmatrix}
		f(y)\\
		(I_3-P(y))c+
		\bigl[(1-t)(|c|^2-1)+t\bigr]P(y)c
	\end{pmatrix}.
	\]
	This is a continuous homotopy from $\mathcal G$ to the map
	\[
	(y,c)\longmapsto(f(y),c).
	\]
	We verify that $\mathcal G_t^+$ has no zeros on
	$\partial(\Theta\times B_2)$. If $y\in\partial\Theta$, then
	the first component of $\mathcal G_t^+$ is nonzero. If
	$|c|=2$, the coefficient of $P(y)c$ in the second component is
	\[
	(1-t)(4-1)+t=3-2t\ge1,
	\]
	and hence the second component is also nonzero. The homotopy
	invariance and product formula for the Brouwer degree therefore
	give
	\begin{equation}
		\begin{aligned}
			\deg(\mathcal G,\Theta\times B_2,0)
			&=\deg\bigl((f(y),c),\Theta\times B_2,0\bigr)\\
			&=\deg(f,\Theta,0)
			\deg(\operatorname{Id},B_2,0)\\
			&=d_f.
		\end{aligned}
		\label{degree on large ball}
	\end{equation}
	
	Next, we compute the degree on $\Theta\times B_{1/2}$.
	Consider the homotopy
	\[
	\mathcal G_t^-(y,c)=
	\begin{pmatrix}
		f(y)\\
		(I_3-P(y))c+
		\bigl[(1-t)(|c|^2-1)-t\bigr]P(y)c
	\end{pmatrix},
	\qquad 0\le t\le1.
	\]
	The first component is nonzero when $y\in\partial\Theta$.
	On $|c|=1/2$, the coefficient of $P(y)c$ is
	\[
	(1-t)\left(\frac14-1\right)-t
	=-\frac34-\frac t4<0.
	\]
	Thus $\mathcal G_t^-$ has no zeros on
	$\partial(\Theta\times B_{1/2})$. At $t=1$,
	\[
	\mathcal G_1^-(y,c)
	=
	\bigl(f(y),R(y)c\bigr),
	\qquad
	R(y)=I_3-2P(y).
	\]
	
	The matrix $R(y)$ reverses the component of a vector in
	$\operatorname{range}P(y)$ and leaves its orthogonal component
	unchanged. Consequently,
	\[
	R(y)^2=I_3,
	\qquad
	|R(y)c|=|c|,
	\qquad
	\det R(y)=-1.
	\]
	Define
	\[
	J(y,c)\coloneqq\bigl(y,R(y)c\bigr).
	\]
	The map $J$ is continuous and satisfies
	\[
	J(J(y,c))
	=
	\bigl(y,R(y)^2c\bigr)
	=
	(y,c).
	\]
	Therefore, $J$ is a homeomorphism and $J^{-1}=J$. Since $J$
	leaves $y$ unchanged and preserves $|c|$, it follows that
	\[
	J\bigl(\Theta\times B_{1/2}\bigr)
	=
	\Theta\times B_{1/2},
	\]
	and similarly,
	\[
	J\bigl(\overline{\Theta}\times\overline{B}_{1/2}\bigr)
	=
	\overline{\Theta}\times\overline{B}_{1/2},
	\qquad
	J\bigl(\partial(\Theta\times B_{1/2})\bigr)
	=
	\partial(\Theta\times B_{1/2}).
	\]
	We now show that the local degree of $J$ is $-1$ at every
	point. Fix
	\[
	(y_0,c_0)\in\Theta\times B_{1/2}
	\]
	and choose an open ball $N$ centered at $(y_0,c_0)$ such that
	\[
	\overline N\subset\Theta\times B_{1/2}.
	\]
	Set
	\[
	\eta_0
	=
	J(y_0,c_0)
	=
	\bigl(y_0,R(y_0)c_0\bigr).
	\]
	On $\overline N$, define
	\[
	J_s(y,c)
	\coloneqq
	\bigl(y,[(1-s)R(y)+sR(y_0)]c\bigr),
	\qquad 0\le s\le1.
	\]
	For every $s\in[0,1]$, if $J_s(y,c)=\eta_0$, then
	$y=y_0$. At this value of $y$, the second component becomes
	$R(y_0)c$, and hence $c=c_0$. Therefore, $(y_0,c_0)$ is the
	unique preimage of $\eta_0$. In particular,
	\[
	\eta_0\notin J_s(\partial N)
	\qquad\text{for every }s\in[0,1].
	\]
	At $s=1$,
	\[
	J_1(y,c)=\bigl(y,R(y_0)c\bigr),
	\]
	whose derivative is
	\[
	\operatorname{diag}(I_m,R(y_0)).
	\]
	Since $\det R(y_0)=-1$, homotopy invariance gives
	\[
	\deg(J,N,\eta_0)
	=
	\deg(J_1,N,\eta_0)
	=
	-1.
	\]
	Since $(y_0,c_0)$ was arbitrarily chosen, we know $J$ has local degree $-1$ everywhere in $\Theta\times B_{1/2}$. Now, define
	\[
	\Phi(y,c)\coloneqq(f(y),c),
	\]
	then
	\[
	\mathcal G_1^-=\Phi\circ J.
	\]
	Since $J$ maps $\Theta\times B_{1/2}$ onto itself and has local degree $-1$ everywhere, the change-of-variables and product formula for the Brouwer degree gives
	\begin{equation}
		\begin{aligned}
			\deg(\mathcal G,\Theta\times B_{1/2},0)
			&=\deg(\Phi\circ J,\Theta\times B_{1/2},0)\\
			&=-\deg(\Phi,\Theta\times B_{1/2},0)\\
			&=-\deg(f,\Theta,0)
			\deg(\operatorname{Id},B_{1/2},0)\\
			&=-d_f.
		\end{aligned}
		\label{top:inner-degree}
	\end{equation}
	The same argument applies when $\Theta$ is disconnected, because the local degree of $J$ is $-1$ on every component.
	
	Finally, the disjoint open sets $\Theta\times B_{1/2}$ and $\Theta\times\mathcal A$ contain all zeros of $\mathcal G$ in $\Theta\times B_2$. Indeed, the sphere $|c|=1/2$ contains no zeros by \eqref{Ep zeros}. By the additivity property of the Brouwer degree, we have
	\[
	\deg(\mathcal G,\Theta\times B_2,0)=\deg(\mathcal G,\Theta\times B_{1/2},0)+\deg(\mathcal G,\Theta\times\mathcal A,0).
	\]
	Substituting \eqref{degree on large ball} and \eqref{top:inner-degree}, we obtain that
	\[
	\deg(\mathcal G,\Theta\times\mathcal A,0)=d_f-(-d_f)=2d_f.
	\]
	
	\textbf{Step 3: A homotopy from $v$ to $-v$.} The equation $E_{P(v,y)}(c)=0$ allows us to retain both unit
	vectors in $\operatorname{range}P(v,y)$ without choosing one of them continuously. Indeed, by \eqref{Ep zeros}, its zeros in $\overline{\mathcal A}$ are precisely the two unit vectors in this one-dimensional range. For either such vector $c$, \eqref{property of P} gives $c\perp v$, and hence
	\[
	s\mapsto \cos(\pi s)v+\sin(\pi s)c,
	\quad 0\le s\le1,
	\]
	is a path in $\mathbb S^2$ from $v$ to $-v$.
	
	To place this construction in a Euclidean domain on which the Brouwer degree is defined, we represent $v\in\mathbb S^2$ as $v=x/|x|$, where $x$ ranges over the annulus $\mathcal A$. Fix $v_*\in\mathbb S^2$ and define
	\begin{equation}
		\mathcal D\coloneqq\left\{
		(x,y,c)\in\mathbb R^3\times\mathbb R^m\times\mathbb R^3:
		x,c\in\mathcal A,\
		\left(\frac{x}{|x|},y\right)\in W
		\right\}.
		\label{aux domain}
	\end{equation}
	The map
	\[
	\Xi:\mathcal A\times U\times\mathcal A
	\rightarrow\mathbb S^2\times U,
	\qquad
	\Xi(x,y,c)=\left(\frac{x}{|x|},y\right),
	\]
	is continuous, and
	\[
	\mathcal D=\Xi^{-1}(W).
	\]
	Since $W$ is relatively open in $\mathbb S^2\times U$, we know $\mathcal D$ is open in $\mathbb R^{m+6}$. It is also bounded because $\mathcal A$ is bounded and $\overline W$ is compact. If $(x,y,c)\in\overline{\mathcal D}$, then
	\[
	\frac12\le |x|,|c|\le2,\quad\left(\frac{x}{|x|},y\right)\in\overline W.
	\]
	In particular, $x\neq0$, so $v=x/|x|$ is well defined on $\overline{\mathcal D}$. Moreover,
	\[
	\overline W
	\subset
	\mathcal O\cap(\mathbb S^2\times Y)
	\subset
	\mathcal O\cap(\mathbb S^2\times U).
	\], hence the maps $P(v,y)$ and $\mathcal F(v,y)$ are also well-defined for $(x,y,c)\in\overline{\mathcal D}$. For $0\le s\le1$, define
	\begin{equation}
		\mathcal H_s(x,y,c)
		\coloneqq
		\begin{pmatrix}
			|x|\bigl(\cos(\pi s)v+
			\sin(\pi s)P(v,y)c\bigr)-v_*\\
			\mathcal F(v,y)\\
			E_{P(v,y)}(c)
		\end{pmatrix},
		\quad
		v=\frac{x}{|x|}.
		\label{final homotopy import}
	\end{equation}
	$\mathcal H_s$ is a continuous map on $[0,1]\times\overline{\mathcal D}$. At the endpoints, we have
	\[
	\mathcal H_0(x,y,c)
	=
	\begin{pmatrix}
		x-v_*\\
		\mathcal F(v,y)\\
		E_{P(v,y)}(c)
	\end{pmatrix},
	\quad
	\mathcal H_1(x,y,c)
	=
	\begin{pmatrix}
		-x-v_*\\
		\mathcal F(v,y)\\
		E_{P(v,y)}(c)
	\end{pmatrix}.
	\]
	It remains to verify that no member of this homotopy vanishes on $\partial\mathcal D$. Suppose that
	\[
	\mathcal H_s(x,y,c)=0
	\]
	for some $s\in[0,1]$ and $(x,y,c)\in\overline{\mathcal D}$. Then we know
	\[
	\mathcal F(v,y)=0.
	\]
	Since $(v,y)\in\overline W\subset\mathbb S^2\times Y$, it follows from the definition of $K_0$ that
	\[
	(v,y)\in K_0\subset W.
	\]
	Since $E_{P(v,y)}(c)=0$, by \eqref{Ep zeros}, we know $P(v,y)c=c$ and $|c|=1$. Since $P(v,y)$ is symmetric and $P(v,y)v=0$, we have
	\[
	v\cdot c
	=
	v\cdot P(v,y)c
	=
	P(v,y)v\cdot c
	=
	0.
	\]
	From $|v|=|c|=1$, there holds
	\begin{align*}
		\left|\cos(\pi s)v+\sin(\pi s)c\right|^2
		&=
		\cos^2(\pi s)|v|^2+\sin^2(\pi s)|c|^2+2\cos(\pi s)\sin(\pi s)v\cdot c\\
		&=
		\cos^2(\pi s)+\sin^2(\pi s)\\
		&=1.
	\end{align*}
	Using $P(v,y)c=c$, we also have
	\[
	|x|\bigl(\cos(\pi s)v+\sin(\pi s)c\bigr)=v_*,
	\]
	from which we know $|x|=1$. Consequently, $x,c\in\mathcal A$, $(v,y)\in W$, and hence $(x,y,c)\in\mathcal D$ by
	\eqref{aux domain}. Therefore, we have proved that every zero of $\mathcal H_s$ in $\overline{\mathcal D}$ lies in the interior $\mathcal D$, and the Brouwer degree are well-defined. Finally, the homotopy invariance yields that
	\begin{equation}
		\deg(\mathcal H_0,\mathcal D,0)=\deg(\mathcal H_1,\mathcal D,0).
		\label{H0 and H1}
	\end{equation}
	
	\textbf{Step 4: The two endpoint degrees.} Fix $\sigma\in\{1,-1\}$. For simplicity, write $\mathcal J_1=\mathcal H_0$ and $\mathcal J_{-1}=\mathcal H_{1}$. If $\mathcal J_\sigma(x,y,c)=0$, then	$x=q_\sigma\coloneqq\sigma v_*$ and hence $v=q_\sigma$.
	The set
	\[
	K_\sigma=\{y\in\overline Y:\mathcal F(q_\sigma,y)=0\}
	\]
	is compact and lies in the open slice
	\[
	W_{q_\sigma}=\{y\in U:(q_\sigma,y)\in W\}\subset Y.
	\]
	Choose a bounded open set $\Theta_\sigma$ such that
	\[
	K_\sigma\subset\Theta_\sigma,\quad\overline{\Theta_\sigma}\subset W_{q_\sigma}.
	\]
	Since $\Theta_\sigma$ contains all zeros of $\mathcal F(q_\sigma,\cdot)$ in $\overline Y$, the excision property of Brouwer degree and the assumption on $\mathcal{F}$ give that
	\begin{equation}
		\deg\bigl(\mathcal F(q_\sigma,\cdot),\Theta_\sigma,0\bigr)=\deg\bigl(\mathcal F(q_\sigma,\cdot),Y,0\bigr)=1.
		\label{fiber degree}
	\end{equation}
	For each $y\in\overline{\Theta_\sigma}$, since $W$ is open, there exist neighborhoods $U_y$ of $q_\sigma$ in $\mathbb S^2$ and $N_y$ of $y$ in $\R^m$ such that
	\[
	U_y\times N_y\subset W.
	\]
	The collection $\{N_y:y\in\overline{\Theta_\sigma}\}$ is an open cover of the compact set
	$\overline{\Theta_\sigma}$. Hence there exist $y_1,\cdots,y_\ell\in\overline{\Theta_\sigma}$ such that
	\[
	\overline{\Theta_\sigma}\subset\bigcup_{j=1}^{\ell}N_{y_j}.
	\]
	Let
	\[
	U_\sigma\coloneqq\bigcap_{j=1}^{\ell}U_{y_j},
	\]
	then $U_\sigma$ is an open neighborhood of $q_\sigma$ in $\mathbb S^2$, and
	\[
	U_\sigma\times\overline{\Theta_\sigma}\subset W.
	\]
	Indeed, if $v\in U_\sigma$ and $y\in\overline{\Theta_\sigma}$, then $y\in N_{y_j}$ for some $j$, while
	$v\in U_\sigma\subset U_{y_j}$. Consequently,
	\[
	(v,y)\in U_{y_j}\times N_{y_j}\subset W.
	\]
	Now consider the continuous normalization map
	\[
	\rho:\mathcal A\rightarrow\mathbb S^2,\quad\rho(x)=\frac{x}{|x|}.
	\]
	Since $|q_\sigma|=1$, we have $q_\sigma\in\mathcal A$ and $\rho(q_\sigma)=q_\sigma$. Since $U_\sigma$ is an open neighborhood of $q_\sigma$, we may choose $\delta>0$ sufficiently small that
	\[
	\overline{B_\delta(q_\sigma)}\subset\mathcal A\cap\rho^{-1}(U_\sigma).
	\]
	It follows that
	\[
	\left(\frac{x}{|x|},y\right)\in W
	\]
	for every $x\in\overline{B_\delta(q_\sigma)}$ and $y\in\overline{\Theta_\sigma}$. Therefore, the open cylinder
	\[
	\mathcal C_\sigma\coloneqq B_\delta(q_\sigma)\times\Theta_\sigma\times\mathcal A
	\]
	is contained in $\mathcal D$. Moreover, every zero of $\mathcal J_\sigma$ in $\mathcal D$ belongs to
	$\mathcal C_\sigma$. Indeed, by $\mathcal J_\sigma(x,y,c)=0$, we know $x=q_\sigma$ and $y\in K_\sigma\subset\Theta_\sigma$. Hence the excision property of the Brouwer degree yields that
	\begin{equation}
		\deg(\mathcal J_\sigma,\mathcal D,0)=\deg(\mathcal J_\sigma,\mathcal C_\sigma,0).\label{qie chu}
	\end{equation}

	Fix $\sigma\in\{1,-1\}$. We now replace $v=x/|x|$ in the last two components of $\mathcal J_\sigma$
	by the fixed vector $q_\sigma$, while keeping the first component unchanged. More precisely, on $\overline{\mathcal C_\sigma}$, define
	\[
	\mathcal J_{\sigma,s}(x,y,c)=
	\begin{pmatrix}
		\sigma x-v_*\\
		(1-s)\mathcal F(v,y)+s\mathcal F(q_\sigma,y)\\
		(1-s)E_{P(v,y)}(c)+sE_{P(q_\sigma,y)}(c)
	\end{pmatrix},
	\quad
	v=\frac{x}{|x|},\quad 0\le s\le1.
	\]
	The choice of $\delta$ and $\Theta_\sigma$ ensures that all terms are defined and continuous on
	$[0,1]\times\overline{\mathcal C_\sigma}$. At $s=0$, this is precisely the original map $\mathcal J_\sigma$, restricted to $\overline{\mathcal C_\sigma}$. Suppose that
	\[
	\mathcal J_{\sigma,s}(x,y,c)=0
	\quad
	\text{for some }(x,y,c)\in\overline{\mathcal C_\sigma},
	\]
	then we know
	\[
	x=\sigma v_*=q_\sigma,
	\quad
	v=\frac{q_\sigma}{|q_\sigma|}=q_\sigma,
	\]
	and the last two components of the equation $\mathcal J_{\sigma,s}(x,y,c)=0$ reduce, for every $s$, to
	\[
	\mathcal F(q_\sigma,y)=0,\quad E_{P(q_\sigma,y)}(c)=0.
	\]
	Since $\overline{\Theta_\sigma}\subset Y$, we know $y\in K_\sigma\subset\Theta_\sigma$ and $c\in\mathcal A$ since $|c|=1$. Since $x=q_\sigma\in B_\delta(q_\sigma)$, we know
	\[
	(x,y,c)\in\mathcal C_\sigma,
	\]
	i.e., every zero of $\mathcal J_{\sigma,s}$ lies in the interior of the cylinder. For $s=1$, we have
	\[
	\mathcal J_{\sigma,1}(x,y,c)
	=
	\bigl(
	\sigma x-v_*,
	\mathcal F(q_\sigma,y),
	E_{P(q_\sigma,y)}(c)
	\bigr).
	\]
	By \eqref{qie chu} and the homotopy invariance, we obtain that
	\[
	\deg(\mathcal J_\sigma,\mathcal D,0)=\deg(\mathcal J_{\sigma,1},\mathcal C_\sigma,0).
	\]
	Notice that the first component of $\mathcal J_{\sigma,1}$ depends only on $x$, and the remaining components depend only on $(y,c)$. We can therefore apply the product formula. More concretely, the map $x\mapsto\sigma x-v_*$ has exactly one zero $q_\sigma$ in $B_\delta(q_\sigma)$, and its derivative is $\sigma I_3$. Hence
	\[
	\deg\bigl(\sigma x-v_*,B_\delta(q_\sigma),0\bigr)=\operatorname{sgn}\det(\sigma I_3)=\sigma.
	\]
	Besides, by \eqref{degree of G on annulus} and \eqref{fiber degree}, we have
	\[
	\begin{aligned}
		&\deg\Bigl(
		\bigl(\mathcal F(q_\sigma,y),E_{P(q_\sigma,y)}(c)\bigr),
		\Theta_\sigma\times\mathcal A,0
		\Bigr)=2\deg\bigl(\mathcal F(q_\sigma,\cdot),
		\Theta_\sigma,0\bigr)
		=2.
	\end{aligned}
	\]
Therefore, by the product formula, we know
	\[
	\deg(\mathcal J_\sigma,\mathcal D,0)=2\sigma.
	\]
	Consequently, since $\mathcal J_{1}=\mathcal H_0$ and $\mathcal J_{-1}=\mathcal H_1$, we have
	\begin{equation}
		\deg(\mathcal H_0,\mathcal D,0)=2,
		\quad
		\deg(\mathcal H_1,\mathcal D,0)=-2,
	\end{equation}
	which is obviously a contradiction to \eqref{H0 and H1}. The proof is complete.
\end{proof}

Now, we are ready to prove Theorem \ref{main result}.

\begin{proof}[Proof of Theorem \ref{main result}]
	Recall that we have normalized $\lambda_4(\Omega)=\alpha^2$ and assumed, for a contradiction, that $\lambda_5(\Omega)>\beta^2$. Set $m=1+2n$ and write $y=(t,A_+,A_-)\in\R^m$. We use the fixed orthonormal coordinates $u_2,u_3,u_4$ to identify $V$ with $\R^3$ and $S(V)$ with $\mathbb S^2$. Let
	\[
	\mathcal F:S(V)\times\R^m\longrightarrow\R^m
	\]
	be the map defined in \eqref{F joint map}. By Lemma \ref{balance center solve degree}, we know $\mathcal{F}$ is continuous, has no zeros on $S(V)\times\partial Y$, and satisfies
	\[
	\deg(\mathcal F(v,\cdot),Y,0)=1,\quad v\in S(V).
	\]
	Moreover, the zero set
	\begin{equation}
		Z=\{(v,t,A_+,A_-)\in S(V)\times\R^{1+2n}:
		\mathcal F(v,t,A_+,A_-)=0\}\nonumber
	\end{equation}
	is compact and contained in $S(V)\times Y$. By Proposition~\ref{neighborhood of Z}, there exists an open neighborhood $\mathcal O$ such that
	\[
	Z\subset\mathcal O\subset S(V)\times Y.
	\]
 Moreover, $\Pi(\zeta)\in\R^{n\times n}$ is a continuous rank-one orthogonal projection in $\mathcal O$ and $T(\zeta)\in\R^{3\times n}$ is continuous in $\mathcal O$. By \eqref{orthogonal condition Pi}, we know
 $$
 v\cdot T(\zeta)c=0,\quad  \forall\, c\in\R^n,	\quad T(\zeta)c\ne0,\quad	\forall\, c\in\operatorname{range}\Pi(\zeta)\setminus\{0\}
 $$
 for any $\zeta=(v,y)\in\mathcal O$. Since $Y$ is open in $\R^m$, the set $\mathcal O$ is also open in $S(V)\times\R^m$. However, these properties contradict Lemma \ref{contradiction obstruction}, applied with
	\[
	m=1+2n,\quad U=\R^m,\quad r=n,\quad K_0=Z.
	\]
Therefore, the assumption $\lambda_5(\Omega)>\beta^2$ is impossible, and we conclude that $\lambda_5(\Omega)\leq\beta^2$.

It remains to prove that the constant in \eqref{main result equation} is optimal. Let $D_0$ be the union of four unit balls with pairwise disjoint closures. Then
\[
\lambda_4(D_0)=\alpha^2,
\qquad
\lambda_5(D_0)=\beta^2.
\]
We join the balls by tubes of radius $\varepsilon$ along fixed line segments to obtain bounded connected
domains $D_\varepsilon$. A standard min-max and compactness argument then gives
\[
\lambda_j(D_\varepsilon)\rightarrow\lambda_j(D_0),
\quad j=4,5,
\]
as $\varepsilon\downarrow0$. Consequently,
\[
\frac{\lambda_5(D_\varepsilon)}
{\lambda_4(D_\varepsilon)}
\rightarrow
\frac{\beta^2}{\alpha^2}=\left(\frac{j_{n / 2,1}}{j_{n / 2-1,1}}\right)^2.
\]
The proof is complete.
\end{proof}

\section{Proof of Theorem \ref{main result for open set}}\label{section6}
Let $\Omega$ be a bounded open set in $\Rn$, $n\geq2$. Its connected components are open and pairwise disjoint. Note that $\Omega$ has at most countably many connected components, since every component contains a point with rational coordinates. The Dirichlet eigenvalues of $\Omega$ are obtained by collecting the Dirichlet eigenvalues of all its connected components and arranging them in non-decreasing order, counting multiplicity. If the same eigenvalue occurs on several components, its multiplicity on $\Omega$ is the sum of its multiplicities on those components. Now, we are ready to establish Theorem \ref{main result for open set}.

\begin{proof}[Proof of Theorem \ref{main result for open set}]
	Combining Theorem \ref{main result} and the excellent results established by Ashbaugh and Benguria \cite{AshbaughBenguria1991BullAMS,AshbaughBenguria1991Duke,AshbaughBenguria1992AnnMath,AshbaughBenguria1992CMP,AshbaughBenguria1993Edinburgh}, we know for every bounded domain $D$,
	\begin{equation}
		\lambda_{r+1}(D)\le \gamma_n\lambda_r(D),\quad \gamma_n=\left(\frac{j_{n / 2,1}}{j_{n / 2-1,1}}\right)^2,
		\quad r=1,2,3,4.\label{result for domain}
	\end{equation}
	Suppose for contradiction that there exists a bounded open set $\Omega$ such that
	\[
	\lambda_5(\Omega)>\gamma_n\lambda_4(\Omega).
	\]
	Denote $L=\lambda_4(\Omega)$ and $T=\gamma_nL$. Since $L>0$ and $\gamma_n>1$, we have
	\begin{equation}
		\lambda_5(\Omega)>T>L.\label{assumption on lambda 5}
	\end{equation}
	Consequently
	$$
	\#\{j\geq 1:\lambda_j(\Omega)\leq L\}=4.
	$$
	Choose a connected component $D$ of $\Omega$ such that $\lambda_1(D)\le L$. Such a component does exist because $\lambda_4(\Omega)=L$ and every eigenvalue of $\Omega$ is an eigenvalue of one of its connected components. Let
	\[
	r\coloneqq\#\{j\ge1:\lambda_j(D)\le L\}.
	\]
	By the definition of $D$, we know
	\[
	1\le r\le4,
	\quad
	\lambda_r(D)\le L<\lambda_{r+1}(D).
	\]
	By \eqref{result for domain}, one has
	\[
	L<\lambda_{r+1}(D)\le \gamma_n\lambda_r(D)\le \gamma_nL=T.
	\]
	Since $D$ is a connected component of $\Omega$,
	$\lambda_{r+1}(D)$ is also a Dirichlet eigenvalue
	of $\Omega$. The preceding estimate gives
	\[
	L<\lambda_{r+1}(D)\le T.
	\]
	On the other hand, since
	\[
	\lambda_j(\Omega)\le L,\qquad j=1,2,3,4,
	\]
	$\lambda_{r+1}(D)$ is strictly larger than the first four eigenvalues of $\Omega$. Therefore, we must have 
	\[
	\lambda_5(\Omega)\le\lambda_{r+1}(D)\le T.
	\]
	However, this is obviously a contradiction to \eqref{assumption on lambda 5}. Finally, let $D_0$ be the union of four unit balls with pairwise disjoint closures. We immediately know that the constant in \eqref{main result equation open set} is optimal. The proof is complete.
\end{proof}

\section*{Data Availability} No data was used for the research described in this paper.

\section*{Statements and Declarations}

\subsection*{Conflict of interest} On behalf of all authors, the corresponding author states that there is no conflict of interest.

\subsection*{AI usage statement} During the preparation of this work the authors used OpenAI model (ChatGPT-6 Astra) as an auxiliary tool for exploring proof strategies and developing mathematical arguments, as well as for language editing and the organization of the exposition. After using this tool, the authors reviewed and edited the content as needed and take full responsibility for all mathematical statements, arguments, citations, and conclusions in this paper.

	\bibliographystyle{amsplain}

\end{document}